\documentclass[11pt,reqno]{amsart}
\usepackage[margin=1.1in]{geometry}

\usepackage{amsmath,amsthm,amssymb}
\usepackage{mathrsfs} 
\usepackage{mathabx}
\usepackage{array}

\usepackage{tikz}
\usepackage{pgfplots}

\usetikzlibrary{arrows}

\usepackage{bm} 
\usepackage{amsmath, amsfonts, amssymb, amsthm} 
\usepackage{style}
\usepackage{enumerate}
\usepackage{mathrsfs}
\usepackage{graphicx} 
\usepackage{float} 
\usepackage{enumitem} 
\usepackage{fancyhdr} 
\usepackage{comment}
\usepackage{caption}  
\usepackage{booktabs} 
\usepackage[comma,sort, numbers]{natbib}
\usepackage{subcaption}

\numberwithin{equation}{section}

\def\R{\mathbb{R}}

\def\N{\mathbb{N}}
\def\P{\mathbb{P}}
\def\E{\mathbb{E}}

\def\Var{\mathrm{Var}}
\def\Cov{\mathrm{Cov}}

\def\Beta{\mathrm{Beta}}

\def\1{\mathbbm{1}}

\allowdisplaybreaks

\title[Thresholds and Fluctuations for Colorful Arithmetic Progressions]{Thresholds and Fluctuations for Colorful Arithmetic Progressions in Sparse Random Colorings }

\author[Bhattacharya]{Bhaswar B. Bhattacharya} 
\address{University of Pennsylvania and National University of Singapore} \email{bhaswar@wharton.upenn.edu}

\author[Bhowal]{Sanchayan Bhowal}
\address{Stanford University} 
\email{sbhowal@stanford.edu}

\author[Sengupta]{Atmadeep Sengupta}
\address{Indian Statistical Institute, Kolkata} 
\email{mb2613@isical.ac.in}

\begin{document}
\begin{abstract}
In this paper, we derive thresholds and fluctuations for arithmetic progressions with prescribed color patterns in sparse random colorings of $[n]:=\{1, 2, \ldots, n\}$, where each element of $[n]$ is colored independently according to a given probability vector. For any admissible ordered palette of colors, we determine the full multi-parameter threshold region for the appearance of a colorful arithmetic progression. The threshold is governed by two competing mechanisms: a global first-moment condition and a local color-availability condition, resulting in a polyhedral satisfiability region, with a piecewise-polyhedral threshold surface. In the satisfiability region we establish asymptotic normality for the number of colorful arithmetic progressions of a given length, with an explicit rate of convergence in Wasserstein distance. On the threshold surface, we identify three distinct asymptotic regimes: Poisson, compound Poisson with mixed Poisson jumps, and compound Poisson with uniform jumps, after an appropriate normalization.  These results provide a complete description of the threshold and fluctuation behavior of general colored arithmetic progressions under sparse random colorings, in a unified  framework that interpolates between classical uncolored/monochromatic progressions in binomial random subsets and multicolored, including rainbow, arithmetic progressions.  
\end{abstract}

\subjclass[2020]{60C05, 11B25, 60F05}

\keywords{Arithmetic progressions, Poisson and Normal approximations, Threshold phenomena, Stein's method.}

\maketitle

\section{Introduction}

The study of arithmetic progressions in subsets of the integers includes some of the central problems in additive combinatorics and Ramsey theory. Celebrated theorems of \citet{Roth1953} and \citet{Szemeredi1975} establish the existence of arithmetic progressions of any fixed length in every subset of the integers with positive upper density. Another classical result is van der Waerden's theorem \cite{vanDerWaerden1927}, which shows that every coloring of a sufficiently long interval of integers contains a {\it monochromatic} arithmetic progression of a prescribed length. Complementary to the notion of monochromatic progressions, in which every term receives the same color, is that of {\it rainbow} progressions, in which all terms receive distinct colors. The systematic study of rainbow arithmetic progressions was initiated by \citet{JungicLichtMahdianNesetrilRadoicic2003} and has since developed into an arithmetic analogue of anti-Ramsey theory (see, for example, \cite{JungicRadoicic2003,AxenovichFonDerFlaass2004,ConlonJungicRadoicic2007,ButlerEtAl2016,BerikkyzySchulteYoung2017,AxenovichMartin2006} among others).

On the probabilistic side, a natural question is to study the fluctuations in the number of arithmetic progressions of a fixed length in a random subset of the integers. Formally, for an integer $r \geq 3$, let $X_r(n,p)$ denote the number of $r$-term arithmetic progressions ($r$-APs) contained in the binomial random subset $[n]_p$, which is a random subset of $[n]:=\{1,2,\ldots,n\}$ where each element is included independently with probability $p\in(0,1)$. It is well known that $X_r(n,p)$ exhibits a threshold at $p^*_{r,n}:=n^{-\frac{2}{r}}$ (see \cite[Example~3.2]{JLR}). More precisely, if $p\ll p^*_{r,n}$, then with high probability $[n]_p$ contains no $r$-AP, while for $p\gg p^*_{r,n}$, it contains at least one $r$-AP with high probability (see Section~\ref{sec:ab} for formal definitions of the asymptotic notation). Moreover, at the threshold (where $p\asymp p^*_{r,n}$), $X_r(n,p)$ is asymptotically Poisson, and  above the threshold (where $p\gg p^*_{r,n}$), $X_r(n,p)$ satisfies a central limit theorem (see \cite[Theorem~1]{barhoumi2019bivariate}). There is also a long line of interesting work on obtaining precise asymptotics for the tail probabilities of $X_r(n,p)$, that is, for the probabilities that $X_r(n,p)$ deviates significantly above or below its typical value \cite{ChatterjeeDembo2016,HarelMoussetSamotij2024,GriffithsKochSecco2023,Warnke2017,BhattacharyaGangulyShaoZhao2020,BrietGopi2020,MoussetNoeverPanagiotouSamotij2020,FizPontiverosGriffithsSeccoSerra2022,lowertail,uppertail}.

In this paper, we study thresholds and fluctuations for arithmetic progressions with general color constraints under sparse random colorings. To this end, given an integer $c \geq 2$, recall that a $c$-coloring of a set $S \subset [n]$ is a map $\phi: S \rightarrow [c] := \{1, \ldots, c\}$, where $\phi(s)$ denotes the color of the element $s \in S$. For $r \geq 3$, an \emph{$r$-palette} is an ordered sequence of colors $\bm a =(a_1,\ldots, a_r)\in [c]^r$. Using an $r$-palette one can encode the color pattern of an $r$-AP as follows: 


\begin{definition}[Colorful arithmetic progressions]
Fix integers $c \geq 2$ and $r \geq 3$. Given a $c$-coloring $\phi_n$ of $[n]$ and an $r$-palette $\bm a$, an $r$-AP $\bm x = (x, x+d,\ldots, x+(r-1)d)$, where $d \geq 1$ and $x + (s -1) d \in [n]$, for all $s \in [r]$, is said to be $\bm a$-{\it colorful} if $\phi_n(x+(s-1)d) = a_s$ for every $s \in[r]$.
\label{defn:rainbow_ap}
\end{definition}

The above definition includes both monochromatic and rainbow arithmetic progressions as special cases, and allows for more general color constraints that interpolate between these two extremes. 
%
%
Our object of interest in this paper is $X_{r, n}(\bm a)$, the number of $\bm a$-colorful $r$-APs in a random coloring of $[n]$, where each element of $[n]$ is assigned a color in $[c]$ independently according to a probability vector $\bm p=(p_1,\ldots,p_c)$. To determine the threshold region for $X_{r, n}(\bm a)$, we consider a polynomially sparse parametrization of the probabilities $p_1,p_2,\ldots,p_{c-1}$ (see \eqref{eq:probabilitypolynomial}) and restrict our attention to admissible palettes, in which the $c$-th color, whose probability is $1-o(1)$, is omitted. This can be naturally interpreted as randomly coloring a random subset of $[n]$, which in the special case $c=2$, reduces to the study of $X_r(n,p)$, the number of (uncolored) $r$-APs in a binomial random subset $[n]_p$ (see Remark \ref{remark:polynomial}). 
Under this parametrization we obtain the following results: 

\begin{itemize} 

\item First, we determine the precise threshold condition under which
$X_{r, n}(\bm a)$ converges in probability to zero or to infinity (Theorem \ref{thm:region}). In particular, there are two distinct mechanisms governing the threshold: (1) a global first-moment condition, which compares the probability that a fixed progression has the prescribed color pattern with the order $n^2$ of the number of $r$-APs in $[n]$, and (2) a color-availability condition, which identifies the scale at which the rarest color appearing in the palette occurs only a bounded number of times in $[n]$.  The resulting satisfiability region, that is, parameter values for which  $X_{r, n}(\bm a)$ diverges to infinity,  is a polyhedral region in the positive orthant of $\R^{c-1}$. Further, the boundary of the satisfiability region (the threshold surface) consists of three components: a hyperplane boundary (that is determined by the global first-moment), a cube boundary (that is determined by the rare-color condition), and their intersection. 

\item Next, we show that for all parameter values in the satisfiability region, $X_{r,n}(\bm a)$, after appropriate centering and scaling, satisfies a central limit theorem. Further, our
result provides an explicit convergence rate for this normal approximation in the Wasserstein distance (see Theorem \ref{thm:Z}). 

\item For parameter values on the boundary of the satisfiability region (the threshold surface), the limiting fluctuations depend on the component of the boundary on which they lie, as summarized below: 

\begin{itemize}

\item For points solely on the first-moment (hyperplane) boundary, $X_{r,n}(\bm a)$ converges to a Poisson distribution. Here, every color occurring in the palette is represented by a diverging number of sites, and the Poisson behavior arises from the rarity of the colorful progressions themselves. An explicit rate of convergence for this Poisson approximation in the Total Variation distance is provided (see Theorem \ref{thm:poisson}). 

\item For points on the intersection of the hyperplane boundary and  rare-color (cube) boundary, $X_{r,n}(\bm a)$ converges to a compound Poisson distribution with rate 1 and mixed Poisson jump distribution (see Theorem \ref{thm:intersectiondistribution}). Here, a color in the palette occurs only $O(1)$ times, and its locations converge, after rescaling, to a Poisson point process on $[0,1]$. Moreover, the joint distribution of the number of colorful $r$-APs whose rare coordinates are pinned at specific locations converges to independent Poisson random variables, which are conditionally independent given the limiting Poisson point process. Consequently, the limiting count is obtained by summing these Poisson contributions over the random locations of the rare color, which gives the compound Poisson distribution. 

\item Finally, for points solely on the cube boundary,  $X_{r,n}(\bm a)$ after an appropriate polynomial normalization converges to a compound Poisson distribution with rate 1 and uniform jump distribution (see Theorem \ref{thm:cubedistribution}). 
 
\end{itemize}

\end{itemize}
Taken together, these results identify the multi-parameter phase diagram for general colored arithmetic progressions in sparse random colorings, and provide a complete description of the limiting fluctuation regimes across the satisfiability region and the threshold boundary.

 \subsection{Related Work}   
\label{sec:related}

Beyond arithmetic progressions, \citet{RueSpiegelZumalacarregui2018} determined thresholds and established Poisson approximation results for the number of solutions to systems of linear equations in binomial random subsets, under suitable non-degeneracy assumptions. More recently, \citet{RueWotzel2022} established asymptotic normality for the number of solutions to broad classes of systems of linear equations in random sets. Extending these results to the colored setting is a natural direction for future work.

There is also an extensive literature on thresholds for van der Waerden-type results in sparse random subsets of the integers. In this setting, one seeks to determine the threshold for $p$ at which, with high probability, every coloring of $[n]_p$ with a fixed number of colors contains a monochromatic arithmetic progression of a prescribed length (see, for example, \cite{RodlRucinski1997,FriedgutHanPersonSchacht2016,Zohar2022}). Recently, \citet{AlvaradoEtAl2025} also proved a sparse random analogue of the canonical van der Waerden theorem, determining the threshold at which a binomial random subset of $[n]$ has the property that every coloring contains either a monochromatic or a rainbow $r$-AP. These results form part of the broader theory of threshold phenomena for extremal properties in sparse random structures (see, for example, the celebrated papers \cite{BaloghMorrisSamotij2015,ConlonGowers2016,hypergraph,discrete} and the references therein). In these results, the coloring of the random host set is arbitrary or adversarial. In contrast, in the present paper the coloring itself is random, and the object of study is the number of arithmetic progressions realizing a prescribed ordered palette.

\subsection{Asymptotic Notation}   
\label{sec:ab}

Throughout the paper, we will use the following asymptotic notations: For two sequences $a_n$ and $b_n$ we will write $a_n =O( b_n)$ if for all $n$ large enough $a_n \leq C _1 b_n$, for some constant $C_1 > 0$. 
We will also use $a_n \lesssim b_n$, $a_n \gtrsim b_n$, and $a_n \asymp b_n$, to, respectively,  denote  $a_n \leq C_1 b_n$, $a_n \geq C_2 b_n$, and $C_2 b_n \leq a_n \leq C_1 b_n$, for $n$ large enough and constants $C_1, C_2 > 0$. Subscripts in the above notation, for example $O_{\square}$ and $\lesssim_{\square}$, will mean that the hidden constants may depend on the subscripted parameters. Moreover, $a_n =o( b_n)$, and $a_n \sim b_n$ mean $a_n/b_n \rightarrow 0$, and $a_n/b_n \rightarrow 1$, as $n \rightarrow \infty$. 
Throughout, $\frac 10=\infty$, by convention.

\section{Statements of Results} 
\label{sec:results}


\subsection{Thresholds for Colorful Arithmetic Progressions}  
\label{sec:thres}

In this section we derive the threshold region for colorful arithmetic progressions. Towards this, define 
$$\Delta_c = \left\{ \bm p = (p_1, p_2, \ldots, p_c): p_a \in [0, 1] \text{ and } \sum_{a=1} ^c p_a = 1 \right \}, $$
which is the probability simplex in $\R^c$. For $\bm p \in \Delta_c$, let $\mu_{\bm p}$ be the discrete probability measure on $[c]$ that assigns mass $p_a$ to the color $a \in [c]$, that is, $\mu_{\bm p}(\{a\})=p_a$, for every $a \in [c]$. A $\mu_{\bm p}$-random $c$-coloring of $[n]$ is a random $c$-coloring of $[n]$ where each element $ i \in [n]$ is independently assigned a color in $[c]$ according to $\mu_{\bm p}$. Given an $r$-palette $\bm a$, denote by $X_{r, n}(\bm a)$ the number of $\bm a$-colorful $r$-APs in $[n]$ in a $\mu_{\bm p}$ random $c$-coloring. To determine parameter values for which $X_{r, n}(\bm a)$ transitions from zero to infinity, one must consider the sparse regime in which the probabilities in $\bm p$ decay to zero with $n$. Specifically, we consider the following polynomial parametrization, which will be convenient for stating and visualizing our results:
\begin{equation}\label{eq:probabilitypolynomial}
p_a=n^{-\theta_a}, \text{ for } a \in[c -1],
\end{equation}
where $\theta_1,\ldots, \theta_{c-1} > 0$, and $p_c=1-\sum_{a=1}^{c-1} p_a \in (0, 1)$, for all sufficiently large $n$. The $c$-th color will be referred to as the {\it ghost color}, and an $r$-palette $\bm a \in [c-1]^r$ that avoids the ghost color will be referred to as an {\it admissible} $r$-palette. 


\begin{remark}[Random colorings of random subsets]  
As mentioned earlier, to obtain the phase diagram for colorful $r$-APs, one must consider coloring probabilities in a sparse regime. However, because of the simplex constraint, not all coloring probabilities can vanish simultaneously. In fact,  under the parametrization \eqref{eq:probabilitypolynomial}, the ghost color $c$ satisfies $p_c=1-o(1)$. Consequently, each occurrence of color $c$ in a fixed palette does not affect the polynomial order of the probability of that palette. We therefore restrict our attention to palettes that do not contain the ghost color. This restriction also has a natural interpretation in terms of random colorings of random subsets of integers. For instance, suppose $A$ is a random subset of $[n]$ obtained by including each element independently with probability $q$, and then each element of $A$ is independently assigned a color $a \in [c-1]$ with probability $q_a$, where $q_1, q_2, \ldots, q_{c-1} \in [0, 1]$ satisfy $\sum_{a=1}^{c-1} q_a = 1$. If we regard the elements of $A^c := [n] \backslash A$ as being assigned the ghost color $c$, then this corresponds to a random $\mu_{\bm p}$-coloring of $[n]$ with parameters
$p_a= q q_a$, for $a \in [c-1]$, and $p_c= 1-q$. Conversely, given a $\mu_{\bm p}$-random coloring of $[n]$, one can obtain a random $(c-1)$-colored subset $A$ of the above form by choosing 
$$q = \sum_{a=1}^{c-1} p_a \quad \text{ and } \quad q_a= \frac{p_a}{q} = \frac{p_a}{\sum_{a=1}^{c-1} p_a}, $$
for $a \in [c-1]$. Hence, counting colorful APs contained in a random set $A$ with arbitrary palettes on $[c-1]$ is equivalent to counting colorful APs in $[n]$ with admissible palettes. The special case $c=2$ is particularly instructive. In this case, there is only one admissible $r$-palette $\bm a=(1,\ldots,1)$. Then the set $A \subset [n]$ of elements colored 1 is precisely a binomial random subset of $[n]$, in which each element is included independently with probability $p_1$. Consequently, when $c=2$, $X_{r,n}(\bm a)$ is simply the number of (uncolored/monochromatic) $r$-APs contained in a binomial random subset of $[n]$, a quantity that  has been extensively studied in a variety of contexts (as mentioned in the Introduction). It is also worth noting that the parametrization \eqref{eq:probabilitypolynomial} suppresses subpolynomial factors. Nevertheless, it provides a convenient way to describe the phase diagram at the level of polynomial exponents, which is sufficient for our purposes.  
\label{remark:polynomial}
\end{remark}

The following result gives the threshold for $X_{r, n}(\bm a)$, for any admissible palette $\bm a$. For this, denote by $S(\bm a) := \{a \in [c-1] : \text{ color } a \text{ appears in the palette } \bm a\}$, 
that is, the set of distinct colors appearing in the palette $\bm a$.

\begin{theorem} 
Fix integers $c \geq 2$ and $r \geq 3$ and suppose $\bm a$ is an admissible $r$-palette. Then the following holds: 
    \begin{equation}\label{eq:XHGn}
       X_{r, n}(\bm a)  \xrightarrow{P} 
       \begin{cases}
           0 & \sum_{a \in S(\bm a)} w_a \theta_{a} > 2 \text{ or } \theta_{\mathrm{max}}>1, \\
           \infty & \sum_{a \in S(\bm a)} w_a \theta_{a} < 2 \text{ and } \theta_{\mathrm{max}}<1,
       \end{cases} 
  \end{equation}  
where $w_a$ is the number of times the color $a \in S(\bm a)$ appears in $\bm a$ and $\theta_{\max} = \max_{a \in S(\bm a)} \theta_{a}$. 
\label{thm:region}
\end{theorem}

The proof of Theorem \ref{thm:region} relies on standard applications of the first and second moment methods (see Section \ref{sec:thresholdpf}). Theorem \ref{thm:region}, in particular, shows that the {\it satisfiability region}, that is, the set of parameter values for which the number of $\bm a$-colorful $r$-APs diverges in probability, is given by the set:
\begin{align}\label{eq:manycopies}
\Delta_{\bm a}^+ = \left \{ (\theta_a)_{a \in S(\bm a)}  \in \R^{|S(\bm a)|}_{ > 0} : \sum_{a \in S(\bm a)} w_a \theta_{a} < 2 \text{ and } \theta_{\mathrm{max}}<1 \right \}  .  
\end{align} 
Note that $\Delta_{\bm a}^+$ is formed by the intersection of finitely many open halfspaces. Hence, $\Delta_{\bm a}^+$ is an open polyhedral region in the positive orthant of $\R^{|S(\bm a)|}$. The {\it unsatisfiability region},
\begin{align*}
\Delta_{\bm a}^- = \left \{ (\theta_a)_{a \in S(\bm a)}  \in \R^{|S(\bm a)|}_{ > 0} : \sum_{a \in S(\bm a)} w_a \theta_{a} > 2 \text{ or } \theta_{\mathrm{max}} > 1 \right \}  ,  
\end{align*}
consists of the parameter values for which the number of $\bm a$-colorful $r$-APs converges to zero in probability.  These two regions are separated by the {\it threshold surface}, which can be described as: 
\begin{align*}
\partial_{\bm a}
=
\left \{
(\theta_a)_{a \in S(\bm a)}
\in \R^{|S(\bm a)|}_{>0} : \left( \sum_{a \in S(\bm a)} w_a \theta_a \leq 2, \theta_{\mathrm{max}} = 1 \right) \text{ or } 
\left(
\sum_{a \in S(\bm a)} w_a \theta_a = 2, \theta_{\mathrm{max}} \leq 1
\right)
\right\}  .  
\end{align*}
In particular, the threshold surface consists of the following components:  
\begin{itemize}

\item $\sum_{a \in S(\bm a)} w_a \theta_a < 2, \theta_{\mathrm{max}} = 1$: This is the intersection of the open halfspace $\sum_{a \in S(\bm a)} w_a \theta_a < 2$ 
with the boundary $\theta_{\mathrm{max}} = 1$ of the unit cube in the positive orthant of $\R^{|S(\bm a)|}$. If this intersection is non-empty, it is a polyhedral surface consisting of finitely many relatively open hyperplane pieces, each of codimension $1$. We will refer to this component of $\partial_{\bm a}$ as the {\it threshold cube boundary}.

\item $\sum_{a \in S(\bm a)} w_a \theta_a = 2, \theta_{\mathrm{max}} < 1$: This is the intersection of the hyperplane $\sum_{a \in S(\bm a)} w_a \theta_a = 2$ with the interior of the unit cube in the positive orthant of $\R^{|S(\bm a)|}$. Hence, if non-empty, it is a relatively open piece of a hyperplane of codimension $1$. This component of $\partial_{\bm a}$ will be referred to as the {\it threshold hyperplane boundary}.

\item $\sum_{a \in S(\bm a)} w_a \theta_a = 2, \theta_{\mathrm{max}} = 1$: This is the intersection of the hyperplane $\sum_{a \in S(\bm a)} w_a \theta_a = 2$ with the boundary of the unit cube in the positive orthant of $\R^{|S(\bm a)|}$. If non-empty, it consists of finitely many polyhedral pieces of codimension $2$. This will be referred to as the {\it threshold intersection locus}.  

\end{itemize}

\begin{remark}
Note that the threshold cube boundary and the threshold intersection
locus can be nonempty only if a color with exponent $1$ appears
exactly once in the palette. To see this, suppose that $\theta_a=1$, for some $a \in S(\bm a)$. On the cube boundary, $w_a \ge2$ would imply $\sum_{a\in S(\boldsymbol a)}w_a\theta_a\ge2$, which is a contradiction. On the intersection locus, if $w_a\ge2$, then either $w_a\ge3$ or $w_a=2$. In the first case, $\sum_{a\in S(\boldsymbol a)}w_a\theta_a\ge3$, which is a contradiction. In the second case, since $r\ge3$, at least one additional coordinate occurs in the palette, and its positive exponent makes $\sum_{a\in S(\boldsymbol a)}w_a\theta_a > 2$. This shows that $w_a=1$. The same argument shows that there is at most one coordinate whose corresponding exponent equals $1$.
\end{remark}

To illustrate the result in Theorem \ref{thm:region} and the regions described above, we consider the following examples:

\begin{example}[APs in binomial random subsets] 
\label{example:subset} 
Suppose $c=2$ and consider the monochromatic $r$-palette $\bm a=(1,1,\ldots,1)$, for $r \geq 3$. In this case, $S(\bm a) = \{1\}$ and $w_1=r$. Further, as mentioned in Remark \ref{remark:polynomial} above, the subset $A \subseteq [n]$ consisting of the elements colored $1$ is a binomial random subset of $[n]$, in which each element is included independently with probability $p_1=n^{-\theta_1}$. Hence, $X_{r,n}(\bm a)$ is precisely the number of $r$-APs contained in $A$. Then, applying Theorem \ref{thm:region}, the satisfiability region is 
$$\Delta_{\bm a}^+ = \left\{ \theta_1>0: r\theta_1<2 \text{ and } \theta_1<1 \right\} = \left\{ 0<\theta_1<\frac{2}{r} \right\}.$$
Moreover, the unsatisfiability region is $\Delta_{\bm a}^- = \{ \theta_1>\frac{2}{r} \}$ and the threshold boundary consists of the single point $\theta_1=\frac{2}{r}$. In particular, this recovers the classical threshold at the polynomial scale for the appearance of an $r$-AP in a binomial random subset of $[n]$.
\end{example}

\begin{example}[Almost monochromatic 4-AP]
\label{example:bichromatic}

Suppose $r=4$, $c=3$, and consider the palette $\bm a=(1,1,1,2)$. In this case, $X_{4,n}(\bm a)$ counts the number of $4$-APs in a random $3$-coloring of $[n]$ in which the first three elements of the AP are colored $1$, while the last element is colored $2$. Here, $S(\bm a)=\{1,2\}$ and 
$w_1=3$, $w_2=1$. Hence, recalling \eqref{eq:manycopies}, the satisfiability region is
\begin{align*}
\Delta^+_{(1,1,1,2)}
&=
\left\{
(\theta_1,\theta_2)\in\R_{>0}^2:
3\theta_1+\theta_2<2,\,
\max\{\theta_1,\theta_2\}<1
\right\} \\
&=
\left\{
(\theta_1,\theta_2)\in\R_{>0}^2:
3\theta_1+\theta_2<2,\,
\theta_2<1
\right\},
\end{align*}
where the last equality follows from the fact that
$3\theta_1+\theta_2<2$, together with $\theta_2>0$, already implies
$0< \theta_1 <\frac{2}{3}$. This is the blue polygonal region shown in Figure
\ref{fig:diagram} (a) , where $X_{4,n}(\bm a)$ diverges in probability.
The threshold boundary is the polyline $(Q_1,Q_2,Q_3)$, excluding the
points $Q_1$ and $Q_3$, with its three components as follows:
\begin{itemize}
    \item The threshold cube boundary is the open line segment
    $(Q_1,Q_2)$ (shown in orange), with endpoints $Q_1=(0,1)$ and $Q_2=(\frac{1}{3}, 1)$. 

    \item The threshold hyperplane boundary is the open line segment
    $(Q_2,Q_3)$ (shown in green), with endpoints $Q_2=(\frac{1}{3}, 1)$  and $Q_3=(\frac{2}{3}, 0)$. 
    
    \item The threshold intersection locus consists of the single point $Q_2=(\frac{1}{3}, 1)$  marked in red.  
\end{itemize}
The remainder of the open positive quadrant, obtained by removing the
satisfiability region and the threshold boundary, is the
unsatisfiability region.  
\end{example}

\begin{figure}[t]
\centering

\begin{minipage}[t]{0.48\textwidth}
\centering

\begin{tikzpicture}[scale=4.25]

        \coordinate (O)  at (0,0);
        \coordinate (Q1) at (0,1);
        \coordinate (Q2) at (1/3,1);
        \coordinate (Q3) at (2/3,0);

        \fill[blue!30] (O) -- (Q1) -- (Q2) -- (Q3) -- cycle;

        \draw[->] (0,0) -- (0.82,0)
            node[right] {$\theta_1$};
        \draw[->] (0,0) -- (0,1.20)
            node[above] {$\theta_2$};

        \draw[orange, very thick] (Q1) -- (Q2);

        \draw[green!60!black, very thick] (Q2) -- (Q3);

        \filldraw[fill=white, draw=orange, thick]
            (Q1) circle (0.5pt);
        \filldraw[fill=white, draw=green!60!black, thick]
            (Q3) circle (0.5pt);

        \fill[red] (Q2) circle (0.5pt);

        \node[left] at (Q1)
            {$Q_1=(0,1)$};

        \node[above right] at (Q2)
            {$Q_2=\left(\frac13,1\right)$};

        \node[below] at (Q3)
            {$Q_3=\left(\frac23,0\right)$};

        \node[below left] at (O)
            {$(0,0)$};

        \node[blue] at (0.26,0.48)
            {$\Delta^+_{(1,1,1,2)}$};

    \end{tikzpicture}
    
    \caption*{(a)}

\end{minipage}
\hfill
\begin{minipage}[t]{0.48\textwidth}
\centering

\begin{tikzpicture}[
    scale=2.85,
    x={(1cm,-0.18cm)},
    y={(-0.65cm,-0.30cm)},
    z={(0cm,1cm)}
]

\coordinate (O) at (0,0,0);
\coordinate (A) at (1,0,0);
\coordinate (B) at (0,1,0);
\coordinate (C) at (0,0,1);
\coordinate (D) at (1,1,0);
\coordinate (E) at (1,0,1);
\coordinate (F) at (0,1,1);

\fill[blue!45, opacity=0.18]
    (O) -- (A) -- (D) -- (B) -- cycle;

\fill[blue!45, opacity=0.18]
    (O) -- (A) -- (E) -- (C) -- cycle;

\fill[blue!45, opacity=0.18]
    (O) -- (B) -- (F) -- (C) -- cycle;

\fill[blue!45, opacity=0.18]
    (A) -- (D) -- (E) -- cycle;

\fill[blue!45, opacity=0.18]
    (B) -- (D) -- (F) -- cycle;

\fill[blue!45, opacity=0.18]
    (C) -- (E) -- (F) -- cycle;

\fill[blue!45, opacity=0.18]
    (D) -- (E) -- (F) -- cycle;


\fill[orange!55, opacity=0.48]
    (A) -- (D) -- (E) -- cycle;

\fill[orange!55, opacity=0.48]
    (B) -- (D) -- (F) -- cycle;

\fill[orange!55, opacity=0.48]
    (C) -- (E) -- (F) -- cycle;


\fill[green!45, opacity=0.58]
    (D) -- (E) -- (F) -- cycle;


\draw[thick] (O) -- (A);
\draw[thick] (O) -- (B);
\draw[thick] (O) -- (C);

\draw[thick] (A) -- (D);
\draw[thick] (A) -- (E);
\draw[thick] (B) -- (D);
\draw[thick] (B) -- (F);
\draw[thick] (C) -- (E);
\draw[thick] (C) -- (F);


\draw[red, very thick] (D) -- (E);
\draw[red, very thick] (E) -- (F);
\draw[red, very thick] (F) -- (D);

\filldraw[fill=white, draw=orange, thick]
    (A) circle (0.014);
\filldraw[fill=white, draw=orange, thick]
    (B) circle (0.014);
\filldraw[fill=white, draw=orange, thick]
    (C) circle (0.014);

\filldraw[fill=white, draw=red, thick]
    (D) circle (0.025);
\filldraw[fill=white, draw=red, thick]
    (E) circle (0.025);
\filldraw[fill=white, draw=red, thick]
    (F) circle (0.025);


\draw[->, thick] (O) -- (1.28,0,0)
    node[right] {$\theta_1$};

\draw[->, thick] (O) -- (0,1.28,0)
    node[left] {$\theta_2$};

\draw[->, thick] (O) -- (0,0,1.28)
    node[above] {$\theta_3$};


\node[blue!70!black] at (0.48,0.25,0.38)
    {$\Delta_{(1,2,3)}^{+}$};

\node[orange!75!black, right] at (1.08,0.095,0.62)
    {$\theta_{\max}=1$};

\node[green!50!black, right] at (0.6,0.42,1.128)
    {$\theta_1+\theta_2+\theta_3=2$};

\end{tikzpicture}

\caption*{(b)}

\end{minipage}

\caption{ \small{ The satisfiability regions and the threshold surfaces for the palettes (a) $\bm a=(1, 1, 1, 2)$ and (b) $\bm a=(1,2,3)$.  }  }  
\label{fig:diagram}

\end{figure}

\begin{example}[Rainbow 3-AP] 
\label{example:differentcolors}
Suppose $r=3$, $c=4$, and consider the palette $\bm a=(1,2,3)$. In this case, $X_{3, n}(\bm a)$ counts the number of rainbow $3$-APs in a random $4$-coloring of $[n]$ in which the first element is colored 1, the second is colored 2, and the third is colored $3$. Here, $S(\bm a)=\{1, 2, 3\}$ and $w_1=w_2=w_3=1$. Hence, the satisfiability region is
\begin{align*}
\Delta_{(1,2,3)}^+ = \left\{ (\theta_1,\theta_2,\theta_3)\in\R_{>0}^3: \theta_1+\theta_2+\theta_3<2, \max\{\theta_1,\theta_2,\theta_3\}<1 \right\}  . 
\end{align*}
Geometrically, this is the region in the open unit cube lying below the plane $\theta_1+\theta_2+\theta_3=2$. The three components of the threshold boundary are as follows (see Figure \ref{fig:diagram} (b)):  
\begin{itemize}
 
\item The threshold cube boundary is the portion of the boundary of the unit cube lying in the halfspace $\theta_1+\theta_2+\theta_3<2$ (shown in orange in Figure \ref{fig:diagram} (b)). 

\item The threshold hyperplane boundary is the portion of the plane $\theta_1+\theta_2+\theta_3=2$ lying inside the open unit cube (shown in green in Figure \ref{fig:diagram} (b)). 

\item The threshold intersection locus is the intersection of the plane $\theta_1+\theta_2+\theta_3=2$ with the boundary of the unit cube. This is the red triangle shown in Figure \ref{fig:diagram} (b), without the points $(1,1,0)$, $(1,0,1)$, and $(0,1,1)$. 

\end{itemize}
The remainder of the open positive orthant, obtained by removing the satisfiability region and the threshold boundary, is the unsatisfiability region.  
\end{example}

\subsection{Asymptotic Normality in the Satisfiability Region}
\label{normal}

Having established the threshold for the emergence of $\bm a$-colorful $r$-APs, the next natural question is to determine the limiting distribution of $X_{r,n}(\bm a)$ in the satisfiability region. For classical (uncolored) APs, it is known that the number of $r$-APs in a binomial random subset is asymptotically normal above the threshold \cite{barhoumi2019bivariate}. In this section, we derive an analogous result for general colorful $r$-APs in the satisfiability regime. Towards this, define 
\begin{align}\label{eq:ZHGn}
Z_{r, n}(\bm a) := \frac{ X_{r, n}(\bm a)  - \E [X_{r, n}(\bm a) ]}{ \sqrt{ \mathrm{Var}[X_{r, n}(\bm a) ] } } . 
\end{align}  
Further, $\mathrm{Wass}(\cdot,\cdot)$ will denote the Wasserstein distance between two random variables, or equivalently between their laws, with a slight abuse of notation depending on the context.  

\begin{theorem}\label{thm:Z}
Fix integers $c \geq 2$ and $r \geq 3$ and suppose $\bm a$ is an admissible $r$-palette. Suppose $\theta_1, \theta_2, \ldots, \theta_{c-1} > 0$ are such that $\sum_{a \in S(\bm a)} w_a \theta_{a} < 2$ and $\theta_{\mathrm{max}} < 1$. Then 
\begin{align}\label{eq:W}
\mathrm{Wass}(Z_{r, n}(\bm a) , \mathcal{N}(0,1) )  \lesssim_{r}  \max\left\{ n^{-1+\frac{1}{2} \sum_{a \in S(\bm a)} w_a \theta_{a}  }  ,  n^{-\frac{1}{2}+\frac{1}{2}  \theta_{\max}  } \right \} . 
\end{align}
In particular, $Z_{r, n}(\bm a) \stackrel{D}  \rightarrow \mathcal{N}(0,1)$.   
\end{theorem}

The proof of Theorem \ref{normal} is given in Section \ref{sec:normalpf}. It uses Stein's method  based on dependency graphs to bound the Wasserstein distance to the normal distribution. In particular, the theorem shows that, for any admissible $r$-palette $\bm a$, the number of $\bm a$-colorful $r$-APs in a random $\mu_{\bm p}$-coloring of $[n]$ is asymptotically normal for all parameter values in the satisfiability region. In other words, whenever the number of $\bm a$-colorful $r$-APs diverges in probability, the appropriately centered and normalized count satisfies a central limit theorem.

\subsection{  Distribution on the Threshold Surface  }
\label{sec:poisson}

In this section, we will analyze the asymptotic behavior of $X_{r, n}(\bm a)$ on the threshold surface.  We begin with the threshold hyperplane boundary, where $\sum_{a \in S(\bm a)} w_a \theta_a=2$ and $\theta_{\mathrm{max}}<1$. In this regime, $\E[X_{r, n}(\bm a)]$ is finite (see \eqref{eq:EXrn}) and the expected number of elements of each color diverges (since $\theta_{\mathrm{max}}<1$ implies, $n p_{a_s} = n^{1-\theta_{a_s}} \gg 1$, for all $s \in [r]$). The following result shows that in this case $X_{r, n}(\bm a)$ converges to a Poisson distribution, for any admissible $r$-palette $\bm a$. In fact, the proof provides an explicit rate of convergence in terms of the  Total Variation distance. The proof uses the Chen-Stein method for Poisson approximation (see Section \ref{sec:poissonpf}). 
Further, $\mathrm{TV}(\cdot,\cdot)$ will denote the Total Variation distance between two random variables, or equivalently between their laws, with a slight abuse of notation depending on the context.

\begin{theorem}
\label{thm:poisson}
Fix integers $c \geq 2$ and $r \geq 3$ and suppose $\bm a$ is an admissible $r$-palette. Suppose $\sum_{a \in S(\bm a)} w_a \theta_a=2$ and $\theta_{\mathrm{max}}<1$. Then $$\mathrm{TV}\left(X_{r, n}(\bm a), \mathrm{Pois}\left(\frac{1}{2(r-1)}\right) \right) \lesssim_r \max\{ n^{-1+\theta_{\max}}, n^{-\theta_{\min}} \}  \ll 1 , $$ 
where $\theta_{\min}:=\min_{a\in S(\bm a)}\theta_a > 0$.  In particular, $X_{r, n}(\bm a)\stackrel{D}  \rightarrow \ \mathrm{Pois}(\frac{1}{2(r-1)})$. 
 \end{theorem}

Next, we consider the threshold intersection locus, where $\sum_{a \in S(\bm a)} w_a \theta_a=2$ and $\theta_{\mathrm{max}}=1$. In this case, $\E[X_{r, n}(\bm a)]$ remains finite (as in the previous case), but at least one color appears only $O_P(1)$ times in the sample. In fact, there is a unique color for which this happens. Formally, writing $\bm a=(a_1,a_2,\ldots,a_r)$, there is a unique index $\hat{s}\in[r]$ such that
\begin{align}\label{eq:location}
\theta_{a_{\hat{s}}}=\theta_{\mathrm{max}}=1. 
\end{align}  
This is because, if there were two distinct indices $s,s'\in[r]$ such that $\theta_{a_s}=\theta_{a_{s'}}=1$, then, since
$\sum_{a\in S(\bm a)}w_a\theta_a\leq 2$, we would have
\begin{equation*}
2
\geq
\sum_{a\in S(\bm a)}w_a\theta_a
=
\sum_{t=1}^r\theta_{a_t}
\geq
2+\sum_{t\in[r]\backslash\{s,s'\}}\theta_{a_t}
>2,
\end{equation*}
where the last inequality follows from $r\geq3$ and $\theta_{a_t}>0$ for every $t\in[r]$, which is a contradiction.  
We will refer to $\hat s$ as the {\it rare coordinate} and $a_{\hat s}$ the {\it rare color}. Also, define the function $\lambda_{\bm a}: [0, 1] \rightarrow \R_{\geq 0}$ as: 
\begin{equation}
\label{eq:lambda}
    \lambda_{\bm a}(x) \coloneq \min\left\{ \frac{x}{\hat{s}-1},\frac{1-x}{r- \hat{s}}\right\}  .  
\end{equation} 
(When one of the terms in the minimum is undefined, $\lambda_{\bm a}(x)$ is defined to be the other term.) Note that, if $U\sim\mathrm{Unif}([0,1])$, then $\lambda_{\bm a}(U) \sim \mathrm{Unif}([0,\frac{1}{r-1}])$.\footnote{To see this, note that for
$0\leq y\leq \frac{1}{r-1}$, $\lambda_{\bm a}(x)>y$ if and only if $(\hat s-1)y<x<1-(r-\hat s)y$.
Hence, $\P(\lambda_{\bm a}(U) >y) = 1-(r-1)y$, for $0\leq y\leq\frac{1}{r-1}$, as required.}  With this definition, we can now state the result about the asymptotic distribution of  $X_{r, n}(\bm a)$ for parameter values on the threshold intersection locus.

\begin{theorem}  
\label{thm:intersectiondistribution}  
Fix integers $c \geq 2$ and $r \geq 3$ and suppose $\bm a$ is an admissible $r$-palette, such that $\sum_{a \in S(\bm a)} w_a \theta_a=2$ and $\theta_{\mathrm{max}}=1$. Then,  
    \begin{equation}\label{eq:intersectiondistribution}
        X_{r, n}(\bm a)  \xrightarrow{D} \sum_{x \in \cP}  N_x  ,  
    \end{equation}
    where $\mathcal P$ is a Poisson point process on $[0,1]$ with rate $1$, and conditional on $\mathcal P$, $\{N_x\}_{x\in\mathcal P}$ are independent with $ N_x\sim \mathrm{Pois}(\lambda_{\bm a}(x)).$  
\end{theorem}

The proof of Theorem \ref{thm:intersectiondistribution} is given in Section \ref{sec:boundarypoissonpf}. First, we identify the random set of indices $\cS_n\subset[n]$ at which the rare color $a_{\hat s}$ appears.
Conditional on $\cS_n=B$, we then show, by a mixed factorial moment calculation, that the joint distribution of the numbers of $\bm a$-colorful $r$-APs whose $\hat s$-th coordinate is pinned at the
locations in $B$ can be approximated by independent Poisson random variables with means $\{ \lambda_{\bm a}(\frac{i}{n})\}_{i\in B}$. Combining this with the fact that $\cS_n$, after an appropriate rescaling, converges to a Poisson point process of rate $1$, gives the marked Poisson point process appearing in Theorem \ref{thm:intersectiondistribution}. The limiting distribution can also be expressed as a compound Poisson distribution, that is, as a Poisson sum of independent random variables, as explained in the following remark.

\begin{remark}
\label{remark:intersectiondistribution}  
For a Poisson process $\cP$ on $[0,1]$ with rate $1$,  the number of points $|\cP|\sim\mathrm{Pois}(1)$ and conditional on $|\cP|=N$, the points of $\cP$ are distributed as independent $\mathrm{Unif}([0,1])$ random
variables $U_1,\ldots,U_N$. Hence, the limiting distribution \eqref{eq:intersectiondistribution} can be expressed as:  
$$Y := \sum_{x \in \cP}  N_x \stackrel{D}{=} \sum_{i=1}^{N} K_i,$$
where $N \sim\mathrm{Pois}(1)$ and $\{K_i\}_{i \geq1}$ are i.i.d. mixed Poisson random variables  (independent of $N$), defined as follows: $$K_i  \sim  \mathrm{Pois}\bigl(\lambda_{\bm a}(U_i)\bigr) \stackrel{D} = \mathrm{Pois} \left ( \mathrm{Unif}\left[0,\frac{1}{r-1}\right] \right )  , $$ 
where $U_i \sim \mathrm{Unif}([0,1])$.\footnote{For a non-negative real-valued random variable $Z$ with law $\cL(Z)$, we denote by $\mathrm{Pois}(\cL(Z))$ (or, interchangeably, $\mathrm{Pois}(Z)$) the law of the random variable $X$ generated as follows: First sample $\Lambda \sim \cL(Z)$, and then, conditional on $\Lambda$, sample $X \sim \mathrm{Pois}(\Lambda)$.} Therefore, the limiting distribution in Theorem \ref{thm:intersectiondistribution} is a
compound Poisson distribution with rate $1$ and mixed Poisson jump
distribution.  
\end{remark}

Next, we consider the threshold cube boundary, where $\sum_{a \in S(\bm a)} w_a \theta_a < 2$ and $\theta_{\mathrm{max}}=1$. Under this assumption, there is again a unique coordinate $\hat s\in[r]$ satisfying $\theta_{a_{\widehat s}}=1$. We define $\lambda_{\boldsymbol a}$ by \eqref{eq:lambda} using this coordinate.
Also, in this case, $\E[X_{r, n}(\bm a)]$ diverges (unlike the previous 2 cases), but one color appears only $O_P(1)$ times in the sample. Hence, in this regime, $X_{r,n}(\bm a)$ scaled by the order of this expectation has a limiting distribution, as shown below.

\begin{theorem}
\label{thm:cubedistribution}
Fix integers $c\geq2$ and $r\geq3$ and suppose $\bm a$ is an admissible
$r$-palette such that $\sum_{a\in S(\bm a)}w_a\theta_a<2$ and $\theta_{\mathrm{max}}=1$. Then
\begin{align}
\label{eq:poissonmark}
\frac{X_{r,n}(\bm a)}
{n^{2-\sum_{a\in S(\bm a)}w_a\theta_a}}
\xrightarrow{D}
\sum_{x\in\cP}\lambda_{\bm a}(x),
\end{align} 
where $\cP$ is a Poisson point process of rate $1$ on $[0,1]$.
\end{theorem}

The proof of Theorem \ref{thm:cubedistribution} is given in Section \ref{sec:cubedistributionpf}. The key step in the proof is to show that, conditional on $\cS_n$ (the set of indices at which the rare color
appears), the number of $\bm a$-colorful $r$-APs whose $\hat s$-th coordinate is pinned at each location in $\cS_n$ concentrates around its expectation.  Then, since $\cS_n$, after an appropriate normalization, converges to a Poisson process of rate 1, we can conclude that $X_{r,n}(\bm a)$, scaled by the order of its expectation, converges to the limit in \eqref{eq:poissonmark}.

\begin{remark} 
\label{remark:cubedistribution}  
Note that the limiting distribution in \eqref{eq:poissonmark} can be equivalently expressed as 
$$Z := \sum_{x\in\cP}\lambda_{\bm a}(x) \stackrel{D}{=} \sum_{i=1}^{N} J_i, $$
where $N\sim\mathrm{Pois}(1)$ and $\{J_i\}_{i \geq 1}$ is a collection of i.i.d. $\mathrm{Unif}([0, \frac{1}{r-1}])$ random variables (independent of $N$). 
Thus, the limiting distribution in \eqref{eq:poissonmark} is a compound Poisson distribution with rate
$1$ and jump distribution $\mathrm{Unif}([0, \frac{1}{r-1}])$. 
The distribution of  $Z$ is also related to the distribution of $Y$ in Remark \ref{remark:intersectiondistribution} as follows: $$Y \stackrel{D} = \mathrm{Pois}\left( \sum_{x\in\cP}\lambda_{\bm a}(x)
\right) \stackrel{D}= \mathrm{Pois}(Z), $$
since conditional on $\cP$, the random variables $\{N_x\}_{x\in\cP}$ are independent Poisson random variables. Thus, the limiting distribution in Theorem \ref{thm:intersectiondistribution} can
equivalently be viewed as a mixed Poisson distribution whose mixing
variable is the compound Poisson limit appearing in Theorem
\ref{thm:cubedistribution}.  
\end{remark}

To illustrate the above results, we now return to the examples in Section \ref{sec:thres}.

\begin{example}[Example \ref{example:subset} continued]  
In this case, $c=2$ and the palette consists of a single color $\bm a=(1,1,\ldots,1)$, for $r \geq 3$. Here, the threshold boundary consists of the single point $\theta_1=\frac{2}{r}$, and by Theorem \ref{thm:poisson}, 
$$X_{r,n}(\bm a) \xrightarrow{D} \mathrm{Pois}\left(\frac{1}{2(r-1)}\right).$$ This recovers the classical Poisson behavior for the number of $r$-term arithmetic progressions in a binomial random subset at the threshold (see, for example, \cite[Theorem 1]{barhoumi2019bivariate}).  
\end{example}

\begin{example}[Example \ref{example:bichromatic} continued]  
In this case, $r=4$, $c=3$, and the palette is $\bm a=(1,1,1,2)$. Here, the threshold boundary is the polyline $(Q_1,Q_2,Q_3)$, as shown in Figure \ref{fig:diagram} (a). Depending on the location of the parameters on this polyline we have the following cases: 

\begin{itemize}  

    \item \textit{Threshold hyperplane boundary}: This is the segment $(Q_2, Q_3)$, where 
    $3\theta_1+\theta_2=2$ and $\frac13<\theta_1<\frac23$.  Then $\theta_{\max}<1$, and Theorem \ref{thm:poisson} gives 
    $$X_{4,n}(\bm a) \xrightarrow{D} \mathrm{Pois}\left(\frac16\right).$$

    \item \textit{Threshold cube boundary}: This is the segment $(Q_1, Q_2)$, where $0<\theta_1< \frac{1}{3}$ and $\theta_2=1$. Note that the rare color is color $2$, which occurs at the fourth coordinate
    of the palette. Hence, $\hat s=4$ and $\lambda_{\bm a}(x)=\frac{x}{3}$. 
    By Theorem \ref{thm:cubedistribution} and Remark \ref{remark:cubedistribution},
    $$\frac{X_{4,n}(\bm a)}{n^{1-3\theta_1}} \xrightarrow{D} \sum_{x\in\cP}\frac{x}{3} \stackrel{D} =  
    \sum_{i=1}^{N}J_i,$$    
    where $\cP$ is a Poisson process of rate 1 on $[0, 1]$, $N\sim\mathrm{Pois}(1)$, and, independently of $N$, $J_1,J_2,\ldots$ are i.i.d.  $\mathrm{Unif}([0, \frac{1}{3}])$. 

    \item \textit{\it Threshold intersection locus}: This is the point $Q_2=(\frac{1}{3},1)$, where $3\theta_1+\theta_2=2$ and $\theta_{\max}=1$. Hence, by Theorem \ref{thm:intersectiondistribution} and Remark \ref{remark:intersectiondistribution},   
    $$X_{4,n}(\bm a) \xrightarrow{D} \sum_{x\in\cP}N_x \stackrel{D}{=}\sum_{i=1}^{N} K_i,$$
    where 
    \begin{itemize} 
    \item $\cP$ is a Poisson process of rate 1 on $[0, 1]$ and, conditional on $\cP$, $\{N_x\}_{x\in \cP}$ are independent with $N_x \sim \mathrm{Pois}(\frac{x}{3})$; 
    
    \item $N \sim\mathrm{Pois}(1)$ and, independently, $\{K_i\}_{i \geq1}$ are i.i.d.\ mixed Poisson random variables with $\mathrm{Pois}(\mathrm{Unif}([0,\frac{1}{3}]))$.  
    \end{itemize} 
\end{itemize}

\end{example}

\begin{example}[Example \ref{example:differentcolors} continued]  
In this case, $r=3$, $c=4$, and the palette is $\bm a=(1,2,3)$. The threshold surface is shown in Figure \ref{fig:diagram} (b). The asymptotic distribution of $X_{r, n}(\bm a)$ on the different components of the threshold surface can be obtained as follows:

\begin{itemize} 

    \item \textit{Threshold hyperplane boundary}: This is the green region in Figure \ref{fig:diagram} (b). In this case, Theorem \ref{thm:poisson} gives  
    $$X_{3,n}(\bm a) \xrightarrow{D} \mathrm{Pois}\left(\frac{1}{4}\right).$$

    \item \textit{Threshold cube boundary}: This is the region shown in orange in Figure \ref{fig:diagram} (b). In this case, depending on the location of the rare color we have, 
\begin{align}\label{eq:lambdadifferentcolors}
\lambda_{\bm a}(x) =
    \begin{cases}
        \dfrac{1-x}{2}, & \text{ if } \theta_1=1,\\ 
        \min\{x,1-x\}, & \text{ if } \theta_2=1,\\
        \dfrac{x}{2}, & \text{ if } \theta_3=1 .
    \end{cases}
\end{align}
    Hence, by Theorem \ref{thm:cubedistribution} and Remark \ref{remark:cubedistribution},
    $$\frac{X_{3,n}(\bm a)}{n^{2-(\theta_1+\theta_2+\theta_3)}} \xrightarrow{D}
    \sum_{x\in\cP}\lambda_{\bm a}(x) \stackrel{D} = \sum_{i=1}^{N} J_i,$$
    where $\cP$ is a Poisson process of rate 1 on $[0, 1]$, $\lambda_{\bm a}(x)$ as defined in \eqref{eq:lambdadifferentcolors}, $N\sim\mathrm{Pois}(1)$ and, independently of $N$, $\{J_i\}_{i \geq 1}$ are i.i.d. $\mathrm{Unif}([0,\frac{1}{2}])$.

    \item \textit{Threshold intersection locus}: This is the red triangle shown in Figure \ref{fig:diagram} (b). In this case, by Theorem \ref{thm:intersectiondistribution} and Remark \ref{remark:intersectiondistribution}, 
    $$X_{3,n}(\bm a) \xrightarrow{D}  \sum_{x\in\cP}N_x \stackrel{D}{=}\sum_{i=1}^{N} K_i , $$
 where 
 \begin{itemize}

   \item $\cP$ is a Poisson process of rate 1 on $[0, 1]$ and, conditional on $\cP$, $\{N_x\}_{x\in \cP}$ are independent with $N_x \sim \mathrm{Pois}(\lambda_{\bm a}(x))$, for $\lambda_{\bm a}(x)$ as in \eqref{eq:lambdadifferentcolors};  

\item $N \sim\mathrm{Pois}(1)$ and, independently of $N$, $\{K_i\}_{i \geq1}$ are i.i.d.\ mixed Poisson random variables with $K_i \sim\mathrm{Pois}(\mathrm{Unif} ([0,\frac12]))$.
\end{itemize}
\end{itemize}  
\end{example}

\begin{remark} 
An interesting future direction is to obtain quantitative rates of convergence for the compound Poisson limits in  Theorems \ref{thm:intersectiondistribution} and \ref{thm:cubedistribution} (similar to those in Theorems \ref{thm:Z} and \ref{thm:poisson}). Techniques from Stein's method for compound Poisson and Poisson process approximation may be useful for this purpose (see, for example, \cite{BarbourChenLoh1992,ChenRollin2013} and the references therein).   
\end{remark}

\section{Proof of Theorem \ref{thm:region}} 
\label{sec:thresholdpf}

Recall that an $r$-AP in $[n]$ is an ordered $r$-tuple $\bm x = (x, x+d,\ldots, x+(r-1)d)$ whose entries lie in $[n]$, where $x \in [n]$ and $d \geq 1$. Denote by $\cX_{r, n}$ the set of $r$-APs in $[n]$. For each fixed $x\in [n]$, every integer $d$ satisfying $1\leq d\leq \frac{n-x}{r-1}$ generates an $r$-AP in $[n]$. Hence,
    \begin{align}\label{eq:Xrn}
       |\cX_{r, n}|
       =\sum_{x=1}^n \left\lfloor\frac{n-x}{r-1}\right\rfloor
       = \sum_{x=1}^n \frac{n-x}{r-1}+O(n)
       = \frac{n^2}{2(r-1)}+O(n).
   \end{align}
Further, for a $\mu_{\bm p}$-random $c$-coloring of $[n]$ and an admissible $r$-palette $\bm a$, the number of $\bm a$-colorful $r$-APs in $[n]$ can be expressed as
\begin{align}\label{eq:Xrncolor}
X_{r,n}(\bm a)
=
\sum_{\bm x \in \cX_{r, n}}
\bm 1 \{\bm x \text{ is } \bm a\text{-colorful} \}.
\end{align}
To begin with, we compute the expectation and the variance of $X_{r, n}(\bm a)$.

\begin{lemma} 
Fix integers $c \geq 2$ and $r \geq 3$ and suppose $\bm a$ is an admissible $r$-palette. Then 
    \begin{align}
    \mathbb{E}[X_{r, n}(\bm a)] \sim \frac{1}{2(r-1)} n^{2-\sum_{ a \in S(\bm a) } w_a \theta_a  } .    
    \label{eq:EXrn}
    \end{align}  
    Further, 
   \begin{align}
\mathrm{Var}[X_{r, n}(\bm a)] \asymp n^{2 - \sum_{ a \in S(\bm a) } w_a \theta_a } + n^{3 - 2\sum_{ a \in S(\bm a) } w_a \theta_a + \theta_{\mathrm{max}} }  .  
  \label{eq:varianceXrn} 
\end{align}    
    \label{lm:expectationvariance} 
\end{lemma} 

\begin{proof}  
For $\bm x \in \cX_{r, n}$ and an admissible $r$-palette $\bm a = (a_1, \ldots, a_r)$,
\begin{align}\label{eq:colorprobability}
\P(\bm x \text{ is } \bm a\text{-colorful})
= \prod_{s=1}^r p_{a_s}
= n^{-\sum_{s=1}^r \theta_{a_s}}
= n^{-\sum_{a \in S(\bm a)} w_a \theta_a}  .  
\end{align}
Combining this with \eqref{eq:Xrn} and \eqref{eq:Xrncolor}, the result in \eqref{eq:EXrn} follows.  

Next, we compute the variance. To this end, for notational convenience denote, 
$$\cE_{\bm a}(\bm x)  = \{ \bm x \text{ is } \bm a\text{-colorful}\}  .  $$
Hence, recalling \eqref{eq:Xrncolor} gives, 
\begin{align}\label{eq:varianceXrnterms}
    \Var[X_{r, n}(\bm a)]  & =  \sum_{\bm x \in \cX_{r, n}} \Var\left[ \bm 1\{\cE_{\bm a}(\bm x)\} \right]  + \sum_{\bm x \ne \bm x' \in \cX_{r, n}} \Cov\left[ \bm 1\{\cE_{\bm a}(\bm x)\}, \bm 1\{\cE_{\bm a}(\bm x')\} \right] .  
\end{align}
From \eqref{eq:colorprobability}, 
$$\Var\left[ \bm 1\{\cE_{\bm a}(\bm x)\} \right]  = 
\P(\cE_{\bm a}(\bm x)) \left( 1-\P(\cE_{\bm a}(\bm x))
\right) \asymp  n^{-\sum_{a \in S(\bm a)} w_a \theta_a} . $$
Hence, recalling \eqref{eq:Xrn}, 
\begin{align}\label{eq:varianceterms}
\sum_{ \bm x \in \cX_{r, n} } \Var\left[ \bm 1\{\cE_{\bm a}(\bm x)\} \right] \asymp n^{2 - \sum_{ a \in S(\bm a) } w_a \theta_a } . 
\end{align}
Now, suppose $\bm x=(x, x+d, \ldots, x + (r-1)d)$ and $\bm x'=(x' , x'+d', \ldots, x'+ (r-1)d')$ be two distinct $r$-APs in $\cX_{r, n}$, for some $x, x' \in [n]$ and $d, d' \geq 1$. Let 
$$\|\bm x - \bm x'\|_0 \coloneqq |\{s \in [r] : x+(s-1) d = x'+(s-1) d' \}|.$$
Note that if $\|\bm x - \bm x'\|_0 \geq 2$, then the two corresponding equalities determine the same arithmetic progression $\bm x=\bm x'$. Hence, we can assume $\|\bm x - \bm x'\|_0 \leq 1$. Also, define 
$$|\bm x \cap \bm x'| = |\{x, x+d, \ldots, x + (r-1)d \} \cap \{x', x'+d', \ldots, x' + (r-1)d'\}|$$ 
which is the number of common points of the two progressions. Now, we consider the following cases:

\begin{itemize}

\item[$(1)$] $|\bm x \cap \bm x'| = 0$: Then $\bm x$ and $\bm x'$ have no elements in common and, hence, 
\begin{align}\label{eq:covdisjoint}
\Cov\left[ \bm 1\{\cE_{\bm a}(\bm x)\}, \bm 1\{\cE_{\bm a}(\bm x')\} \right]= 0 . 
\end{align}

\item[$(2)$] $|\bm x \cap \bm x'| \geq 2$: Since $\bm x\neq \bm x'$, we have $\|\bm x-\bm x'\|_0\leq 1$. Hence, there exists at least one pair $(s,s')$, with $s\neq s'$ in $[r]$, such that $x+(s-1)d=x'+(s'-1)d'$. First, suppose there exists such a pair $(s, s')$ for which $a_s\neq a_{s'}$, that is, the two progressions require different colors at a common point. In this case, $\P(\cE_{\bm a}(\bm x')\cap \cE_{\bm a}(\bm x))=0$, and hence, by \eqref{eq:colorprobability},
\begin{equation}\label{eq:coloras}
\Cov\left[\bm 1\{\cE_{\bm a}(\bm x)\},\bm 1\{\cE_{\bm a}(\bm x')\}\right] = -\P(\cE_{\bm a}(\bm x'))\P(\cE_{\bm a}(\bm x)) = -n^{-2\sum_{a\in S(\bm a)}w_a\theta_a}.
\end{equation}  
Next, suppose that $a_s=a_{s'}$ for every pair $(s,s')$, with $s\neq s'$ in $[r]$, such that $x+(s-1)d=x'+(s'-1)d'$. In this case, the color requirements of the two progressions are compatible at every common point. Since $\bm x\neq\bm x'$, the progression $\bm x'$ contains at least one point that does not belong to $\bm x$. Conditional on $\cE_{\bm a}(\bm x)$, this new point must receive one of the colors in $S(\bm a)$. Therefore,
$$\P(\cE_{\bm a}(\bm x')|\cE_{\bm a}(\bm x)) \leq n^{-\theta_{\min}},$$
where $\theta_{\min}:=\min_{a\in S(\bm a)}\theta_a$. Consequently,
\begin{align*}
\Cov\left[\bm 1\{\cE_{\bm a}(\bm x)\},\bm 1\{\cE_{\bm a}(\bm x')\}\right] \leq
\P(\cE_{\bm a}(\bm x)\cap\cE_{\bm a}(\bm x')) &=
\P(\cE_{\bm a}(\bm x))
\P(\cE_{\bm a}(\bm x')|\cE_{\bm a}(\bm x))\\
&\leq
n^{-\sum_{a\in S(\bm a)}w_a\theta_a-\theta_{\min}}.
\end{align*}
Moreover, for each fixed $\bm x\in\cX_{r, n}$, the number of choices of $\bm x'\neq\bm x$ such that $|\bm x\cap\bm x'|\geq2$ is $O_r(1)$. This is because, one can choose two distinct common points of $\bm x$ and their two positions in $\bm x'$ in only $O_r(1)$ ways, and these choices uniquely determine $\bm x'$. Hence, by \eqref{eq:Xrn}, the number of ordered pairs $(\bm x,\bm x')$ with $\bm x\neq\bm x'$ and $|\bm x\cap\bm x'|\geq2$ is $O_r (n^2)$.  
Hence, 
\begin{align}\label{eq:covtwopoint}
\sum_{\substack{\bm x\neq\bm x'\in\cX_{r, n}\\ |\bm x\cap\bm x'|\geq2}}
\left|
\Cov\left[\bm 1\{\cE_{\bm a}(\bm x)\},\bm 1\{\cE_{\bm a}(\bm x')\}\right]
\right|  \lesssim_r n^{2-\sum_{a\in S(\bm a)}w_a\theta_a-\theta_{\min}}  \ll  n^{2-\sum_{a\in S(\bm a)}w_a\theta_a}  ,  
\end{align}
since $\sum_{a\in S(\bm a)}w_a\theta_a\geq\theta_{\min}$ and $\theta_{\min}>0$.

\item[$(3)$] 

Suppose $|\bm x \cap \bm x'|=1$: Then there exists exactly one pair $(s, s')$, with $s, s' \in [r]$ (which may or may not be equal), such that $x+ (s-1) d = x' + (s'-1) d'$. As in the previous case, if $a_s \ne a_{s'}$, then $\P(\cE_{\bm a}(\bm x') \cap \cE_{\bm a}(\bm x))=0$ and the covariance is given by \eqref{eq:coloras}. Now, suppose $a_s=a_{s'} = b$, for some $b \in S(\bm a)$. Conditional on $\cE_{\bm a}(\bm x)$, the common point $x+ (s-1) d = x' + (s'-1) d'$ already has the color required at the $s'$-th position of $\bm x'$. Therefore, for $\cE_{\bm a}(\bm x')$ to occur, only the remaining $r-1$ color requirements along $\bm x'$ must be realized. Since $|\bm x\cap\bm x'|=1$, these remaining points do not belong to $\bm x$ and their colors are independent of $\cE_{\bm a}(\bm x)$. Hence, 
\begin{equation*}
    \P(\cE_{\bm a}(\bm x')| \cE_{\bm a}(\bm x))
    =
    n^{-\sum_{a \in S(\bm a)} w_a \theta_a+\theta_{b}}.  
\end{equation*}
Consequently, 
\begin{align*}
    \Cov\left[ \bm 1\{\cE_{\bm a}(\bm x)\}, \bm 1\{\cE_{\bm a}(\bm x')\} \right] 
    & = n^{-2 \sum_{a \in S(\bm a)} w_a \theta_a + \theta_{b}}  
    - n^{-2 \sum_{a \in S(\bm a)} w_a \theta_a }  \\  
    & \asymp n^{-2 \sum_{a \in S(\bm a)} w_a \theta_a + \theta_{b}},
\end{align*}
where the last relation follows since $\theta_b>0$. Note that the number of pairs $\bm x , \bm x' \in \cX_{r, n}$ with $|\bm x \cap \bm x'|=1$ is $\lesssim_r n^3$. This is because there are $O(n)$ choices for the common point and $O(n)$ choices for each of the two common differences. Also, note that the contribution to the covariance from pairs $\bm x , \bm x' \in \cX_{r, n}$, with $|\bm x \cap \bm x'|=1$ and the two required colors at the common point are not compatible is 
$$O_r(n^{3-2\sum_{a\in S(\bm a)}w_a\theta_a}) = o(n^{3 - 2\sum_{a \in S(\bm a)} w_a \theta_a + \theta_{\max}}) . $$ 
Combining the above gives, 
\begin{align}
\sum_{\substack{\bm x\neq\bm x'\in\cX_{r, n}\\ |\bm x\cap\bm x'|=1 }}
\Cov\left[
\bm 1\{\cE_{\bm a}(\bm x)\},
\bm 1\{\cE_{\bm a}(\bm x')\}
\right]
\asymp_r
n^{3 - 2\sum_{a \in S(\bm a)} w_a \theta_a + \theta_{\max}}.
\label{eq:covonepoint}
\end{align}
For the lower bound in \eqref{eq:covonepoint}, choose a color $b\in S(\bm a)$ with $\theta_b=\theta_{\max}$ and a position $s \in [r]$ for which $a_s=b$, and consider pairs whose unique common point occurs in position $s$ in both progressions. There are $\gtrsim_r n^3$ such pairs, and each contributes $\asymp n^{-2 \sum_{a \in S(\bm a)} w_a \theta_a+\theta_{\max}}$, which gives the desired lower bound.

\end{itemize}

Combining \eqref{eq:varianceXrnterms}, \eqref{eq:varianceterms}, \eqref{eq:covdisjoint}, \eqref{eq:covtwopoint}, and \eqref{eq:covonepoint}, the result in \eqref{eq:varianceXrn} follows.  
\end{proof}

Using the above lemma, we can now complete the proof of Theorem \ref{thm:region}. First, suppose that $\sum_{a \in S(\bm a)} w_a \theta_a>2$. Then, by \eqref{eq:EXrn},
$\mathbb{E}[X_{r,n}(\bm a)] \ll 1$, and hence, by Markov's inequality,
$$\P(X_{r,n}(\bm a) > 0) \leq \mathbb{E}[X_{r,n}(\bm a)] \ll 1. $$
Next, suppose $\theta_{\mathrm{max}}>1$. Choose $s \in [r]$ such that 
$\theta_{a_s}=\theta_{\mathrm{max}}$. Let $Y_n= |\{ i \in [n]: \phi_n(i) = a_s\}|,$
where $\phi_n$ is a $\mu_{\bm p}$-random coloring of $[n]$. Note that
$\{X_{r,n}(\bm a)>0\} \subseteq \{Y_n>0\}$, and hence,
$$
\P(X_{r,n}(\bm a)>0)
\leq \P(Y_n>0)
\leq \mathbb{E}[Y_n]
= n^{1-\theta_{a_s}}
= n^{1-\theta_{\mathrm{max}}} \ll 1.
$$
This completes the proof of the $0$-statement in \eqref{eq:XHGn}.

To prove the $1$-statement, suppose that
$\sum_{a \in S(\bm a)} w_a \theta_a<2$ and $\theta_{\mathrm{max}}<1$. In this case, by \eqref{eq:EXrn}, $\mathbb{E}[X_{r,n}(\bm a)] \gg 1$. Further, recalling \eqref{eq:varianceXrn},
\begin{align*}
    \frac{\Var[X_{r,n}(\bm a)]}{\mathbb E[X_{r,n}(\bm a)]^2}
    &\lesssim_r 
    \frac{
    n^{2-\sum_{a \in S(\bm a)} w_a \theta_a}
    +
    n^{3-2\sum_{a \in S(\bm a)} w_a \theta_a+\theta_{\mathrm{max}}}
    }{
    n^{4-2\sum_{a \in S(\bm a)} w_a \theta_a}
    }\\
    &\lesssim_r
    n^{\sum_{a \in S(\bm a)} w_a \theta_a-2}
    +
    n^{\theta_{\mathrm{max}}-1}
    \ll 1 , 
\end{align*}
Therefore, by Chebyshev's inequality, $\frac{X_{r,n}(\bm a)}{\mathbb E[X_{r,n}(\bm a)]}\stackrel{P}{\rightarrow} 1$. Since $\mathbb{E}[X_{r,n}(\bm a)] \gg 1$, it follows that $X_{r,n}(\bm a)\stackrel{P}{\rightarrow}\infty$. This completes the proof of Theorem \ref{thm:region}.  \hfill $\Box$

\section{ Proof of Theorem \ref{normal} }
\label{sec:normalpf}

The proof of Theorem \ref{normal} will use Stein's method based on  dependency graphs. We begin by defining the notion of a dependency graph.

\begin{definition}[Dependency graph]
    Let $\{T_v\}_{v\in V}$ be a family of random variables (defined on a common probability space) indexed by a finite set $V$. A graph $\mathcal G$ with vertex set $V$ is said to be a dependency graph for the collection $\{T_v\}_{v\in V}$ if the following holds: For any two disjoint subsets $A, B \subseteq V$ such that there is no edge in $\mathcal{G}$ joining a vertex in $A$ to a vertex in $B$, the collections of random variables $\{T_v\}_{v\in A}$ and $\{T_v\}_{v\in B}$ are independent.
\end{definition}

For a dependency graph $\mathcal G$ and $u \in V$, denote by $\overline{N}_\mathcal{G}(u)$ the set consisting of the neighbors of $u$ in $\mathcal{G}$ together with $u$ itself. Also, denote $\overline{N}_\mathcal{G}(u,v):=\overline{N}_\mathcal{G}(u)\cup\overline{N}_\mathcal{G}(v)$. We will use the following version of Stein's method based on dependency graphs.

\begin{theorem}[{\cite[Theorem 6.33]{JLR}}]
    \label{thm:W}
    Let $\{T_v\}_{v\in V}$ be a family of random variables with dependency graph $\mathcal{G}$. Assume $\E[T_v]=0$ for all $v\in V$. Let $W=\frac{1}{\sigma}\sum_{v\in V}T_v$,  where $\sigma^2=\Var\left[\sum_{v\in V}T_v\right]>0$. Suppose that
    \begin{equation*}
        \sum_{w\in V}\E[|T_w|] \le R
        \quad\text{and}\quad
        \sum_{w\in\overline{N}_\mathcal{G}(u,v)}
        \E[|T_w|| T_u,T_v] \le Q
    \end{equation*}
    almost surely, for all $u,v\in V$. Then, for $Z \sim N(0,1)$, $\mathrm{Wass}(W,Z)\lesssim \frac{RQ^2}{\sigma^3}$.  
\end{theorem}

We can now proceed with the proof of Theorem \ref{normal}. To begin with, for $\bm x \in \cX_{r, n}$, define 
\begin{align*} 
Y_{\bm x} = \bm{1}\{ \cE_{\bm a}(\bm x)\} - \mathbb{E}[\bm{1}\{ \cE_{\bm a}(\bm x)\}] = \bm{1}\{ \cE_{\bm a}(\bm x)\}  -  \mu  , 
\end{align*}
where
$\mu=\mathbb{P}(\cE_{\bm a}(\bm x)) =n^{-\sum_{a \in S(\bm a)} w_a \theta_a}$
(recall \eqref{eq:colorprobability}), and $\sigma^2=\Var\left[\sum_{\bm x \in \cX_{r, n}}Y_{\bm x}\right] = \Var[X_{r,n}(\bm a)]$. 
Then, recalling \eqref{eq:ZHGn},
$$
Z_{r,n}(\bm a)
=
\frac{1}{\sigma}
\sum_{\bm x\in\cX_{r, n}}Y_{\bm x}.
$$
Construct a dependency graph $\mathcal G$ for the collection of random variables
$\{Y_{\bm x}:\bm x\in \cX_{r, n} \}$ with vertex set $\cX_{r, n}$ as follows: Connect two distinct vertices $\bm x,\bm x' \in \cX_{r, n}$ by an edge if and only if $|\bm x\cap\bm x'|\geq 1.$
This is a dependency graph since $Y_{\bm x}$ depends only on the colors assigned to the points of $\bm x$, and collections of APs with no common points depend on disjoint sets of independently assigned colors.

To apply Theorem \ref{thm:W}, first note that $\E|Y_{\bm x}| = 2\mu(1-\mu) \leq 2\mu$. Hence, from \eqref{eq:Xrn},
\begin{align}
\sum_{\bm x\in\cX_{r, n}}\E|Y_{\bm x}|
&\lesssim |\cX_{r, n}|\mu
\lesssim
n^{2-\sum_{a\in S(\bm a)}w_a\theta_a}.
\label{eq:Y}
\end{align}
Next, fix $\bm x',\bm x''\in\cX_{r, n}$ and consider
$\bm x\in\overline{N}_{\mathcal G}(\bm x',\bm x'')$. Write $\bm x=(x,x+d,\ldots,x+(r-1)d) $
and define
$$ J(\bm x | \bm x',\bm x'') := \{s\in[r]:x+(s-1)d\in\bm x'\cup\bm x''\}.$$
Since $Y_{\bm x}= \bm{1}\{ \cE_{\bm a}(\bm x)\}  -  \mu$, 
\begin{align*}
\E[|Y_{\bm x}|| Y_{\bm x'},Y_{\bm x''}] \leq \E[ \bm{1}\{ \cE_{\bm a}(\bm x)\} | Y_{\bm x'},Y_{\bm x''}]+\mu = \P(\cE_{\bm a}(\bm x)| Y_{\bm x'},Y_{\bm x''})+\mu.
\end{align*}
The conditioning on $Y_{\bm x'}$ and $Y_{\bm x''}$ depends only on the colors assigned to the points in $\bm x'\cup\bm x''$. Therefore, the colors of the points of $\bm x$ lying outside $\bm x'\cup\bm x''$ remain independent and retain their original distributions. Consequently,
\begin{align*}
\P(\cE_{\bm a}(\bm x)| Y_{\bm x'},Y_{\bm x''}) \leq \prod_{s\notin J(\bm x | \bm x',\bm x'')}p_{a_s} = \mu \prod_{s\in J(\bm x | \bm x',\bm x'')}p_{a_s}^{-1} =
\mu n^{\sum_{s\in J(\bm x | \bm x',\bm x'')}\theta_{a_s}}.
\end{align*}
We now consider the following two cases: 
\begin{itemize}

\item $|J(\bm x | \bm x',\bm x'')|=1$: Then
\begin{align*}
\E[|Y_{\bm x}|| Y_{\bm x'},Y_{\bm x''}]
&\leq
\mu n^{\theta_{\max}}+\mu
\lesssim
\mu n^{\theta_{\max}}.
\end{align*}
For fixed $\bm x'$ and $\bm x''$, there are $O_r(n)$ choices of $\bm x$ satisfying
$|J(\bm x | \bm x',\bm x'')|=1$. This is because, there are $O_r(1)$ choices for the common point and its position in $\bm x$, and then $O_r(n)$ choices for the common difference of $\bm x$. Hence,
\begin{align*}
\sum_{\substack{\bm x\in\overline{N}_{\mathcal G}(\bm x',\bm x'')\\
|J(\bm x | \bm x',\bm x'')|=1}}
\E[|Y_{\bm x}|| Y_{\bm x'},Y_{\bm x''}]
\lesssim_r
n\mu n^{\theta_{\max}}.
\end{align*}

\item $|J(\bm x | \bm x',\bm x'')|\geq 2$: 
Note that two points of $\bm x'\cup\bm x''$, together with their positions in $\bm x$, determine $\bm x$ uniquely. Hence, in this case, there are only $O_r(1)$ choices of $\bm x$. Also, since $|Y_{\bm x}|\leq 1$,
\begin{align*}
\sum_{\substack{\bm x\in\overline{N}_{\mathcal G}(\bm x',\bm x'')\\
|J(\bm x | \bm x',\bm x'')|\geq 2}}
\E[|Y_{\bm x}|| Y_{\bm x'},Y_{\bm x''}]
\lesssim_r 1.
\end{align*} 
\end{itemize}
Combining the last two displays gives, uniformly over $\bm x',\bm x''\in\cX_{r, n}$,
\begin{align}
\sum_{\bm x\in\overline{N}_{\mathcal G}(\bm x',\bm x'')}
\E[|Y_{\bm x}|| Y_{\bm x'},Y_{\bm x''}] \lesssim_r
1+n\mu n^{\theta_{\max}} = 1+n^{1-\sum_{a\in S(\bm a)}w_a\theta_a+\theta_{\max}}.
\label{eq:neighborhood} 
\end{align}

Using the bounds \eqref{eq:Y}, \eqref{eq:neighborhood}, and \eqref{eq:varianceXrn} in Theorem \ref{thm:W} gives
\begin{align}
\mathrm{Wass}(Z_{r, n}(\bm a) , \mathcal N(0, 1)) 
& \lesssim_r \frac{ n^{2-\sum_{a\in S(\bm a)}w_a\theta_a} \left( 1+ n^{2 - 2 \sum_{a\in S(\bm a)}w_a\theta_a + 2 \theta_{\max}} \right)  }{ \left( n^{2 - \sum_{ a \in S(\bm a) } w_a \theta_a } + n^{3 - 2\sum_{ a \in S(\bm a) } w_a \theta_a + \theta_{\mathrm{max}} }\right)^\frac{3}{2} } .
\label{eq:wassbound}
\end{align}
For notational convenience, let $\Lambda_{\bm a}:=\sum_{a\in S(\bm a)}w_a\theta_a$.
Recall that, in the satisfiability region \eqref{eq:manycopies}, $\Lambda_{\bm a}<2$ and $\theta_{\max}<1$. First, suppose $1-\Lambda_{\bm a}+\theta_{\max}\leq 0$, or equivalently,
$\Lambda_{\bm a}\geq 1+\theta_{\max}$. Then $1+n^{2-2\Lambda_{\bm a}+2\theta_{\max}}\lesssim 1$, and the first term in \eqref{eq:varianceXrn} dominates the second. Hence, \eqref{eq:wassbound} gives,  
\begin{align}\label{eq:Zrnalambda}  
\mathrm{Wass}(Z_{r, n}(\bm a) , \mathcal N(0, 1)) & \lesssim_r \frac{n^{2-\Lambda_{\bm a}}} {n^{\frac{3}{2}(2-\Lambda_{\bm a})}} =
n^{-1+\frac{\Lambda_{\bm a}}{2}}  ,  
\end{align}
since $\Lambda_{\bm a}<2$. Next, suppose $1-\Lambda_{\bm a}+\theta_{\max}>0$, or equivalently,
$\Lambda_{\bm a}<1+\theta_{\max}$. Then $1+n^{2-2\Lambda_{\bm a}+2\theta_{\max}} \lesssim n^{2-2\Lambda_{\bm a}+2\theta_{\max}}$, and the second term in \eqref{eq:varianceXrn} dominates the first. Therefore,
\begin{align}\label{eq:Zrnamax}
\mathrm{Wass}(Z_{r, n}(\bm a) , \mathcal N(0, 1)) \lesssim_r \frac{ n^{4-3\Lambda_{\bm a}+2\theta_{\max}}}{n^{\frac{3}{2}(3-2\Lambda_{\bm a}+\theta_{\max})}} & = n^{-\frac{1}{2}+\frac{\theta_{\max}}{2}} . 
\end{align}
since $\theta_{\max}<1$. Combining \eqref{eq:Zrnalambda} and \eqref{eq:Zrnamax}, the result in \eqref{eq:W} follows. Thus, throughout the satisfiability region, $
\mathrm{Wass}(Z_{r, n}(\bm a) , \mathcal N(0, 1)) \rightarrow 0$. This completes the proof of Theorem \ref{normal}.  \hfill $\Box$

\section{ Proof of Theorem \ref{thm:poisson} }
\label{sec:poissonpf}

We will prove Theorem \ref{thm:poisson} using the Chen-Stein method for Poisson approximation based on dependency graphs \cite{poissonapproximation,poissondependency}.     

\begin{theorem}[{\cite[Theorem~15]{poissonapproximation}}]  
Suppose $\{T_\alpha:\alpha\in A\}$ be a collection of Bernoulli random variables and 
$T:=\sum_{\alpha\in A} T_\alpha$. Let $\mathcal G$ be a dependency graph for this collection.  Then, with $\lambda:=\mathbb E[T]$, the following holds: 
$$\mathrm{TV} \left( T, \operatorname{Pois}(\lambda) \right) \leq
\min\left\{1, \frac{1}{\lambda} \right\} \left( B_1+B_2 \right) , $$ 
where 
$$B_1 := \sum_{\alpha\in A} \sum_{\beta\in \overline{N}_{\mathcal{G}}(\alpha)}
\mathbb E[T_\alpha]\mathbb E[T_\beta] \text{ and } B_2 :=
\sum_{\alpha \in A}
\sum_{\substack{\beta\in \overline{N}_{\mathcal{G}}(\alpha)\\\beta\neq\alpha}}
\mathbb E[T_\alpha T_\beta].$$ 
\label{thm:poissondependencygraph} 
\end{theorem}

As in the proof of Theorem \ref{normal}, construct a dependency graph $\mathcal G$ for the collection of random variables
$\{\bm 1\{\cE_{\bm a}(\bm x)\} :\bm x\in \cX_{r, n} \}$ with vertex set $\cX_{r, n}$ as follows: Connect two distinct vertices $\bm x,\bm x'\in\cX_{r, n}$ by an edge if and only if $|\bm x\cap\bm x'|\geq 1.$ This is a dependency graph since the events corresponding to disjoint collections of APs depend on disjoint sets of independently assigned colors.

Recall that
$$
\mu=\mathbb{P}(\cE_{\bm a}(\bm x))
=n^{-\sum_{a \in S(\bm a)} w_a \theta_a}
=n^{-2},
$$
where the last equality follows from the assumption
$\sum_{a \in S(\bm a)} w_a \theta_a=2$. Hence, by \eqref{eq:Xrn},
\begin{align}
\lambda_n:=\E[X_{r,n}(\bm a)]
=|\cX_{r, n}|\mu = \frac{1}{2(r-1)} + O_r\left(\frac{1}{n}\right).
\label{eq:lambdapoisson}
\end{align}
We now bound the quantities $B_1$ and $B_2$ in Theorem \ref{thm:poissondependencygraph}. First, for every fixed $\bm x\in\cX_{r, n}$, note that the number of $\bm x'\in\cX_{r, n}$ such that $|\bm x\cap\bm x'|\geq 1$ is $O_r(n)$. Therefore,
\begin{align}\label{eq:poissonmain}
B_1 &=
\sum_{\bm x\in\cX_{r, n}}
\sum_{\bm x'\in\overline{N}_{\mathcal G}(\bm x)}
\E[\bm 1\{\cE_{\bm a}(\bm x)\}]
\E[\bm 1\{\cE_{\bm a}(\bm x')\}] \lesssim_r
|\cX_{r, n}| n \mu^2 \lesssim_r n^{-1} . 
\end{align} 
Next, we analyze $B_2$. We consider the following cases: 

\begin{itemize} 

\item $|\bm x\cap\bm x'|=1$: Let the unique common point occur at position $s$ in $\bm x$ and position $s'$ in $\bm x'$. If $a_s\neq a_{s'}$, then the two coloring requirements at the common point are incompatible, and hence, $\P(\cE_{\bm a}(\bm x)\cap\cE_{\bm a}(\bm x'))=0$.
On the other hand, if $a_s=a_{s'}=b$, for some $b\in S(\bm a)$, then
\begin{align*}
\P(\cE_{\bm a}(\bm x)\cap\cE_{\bm a}(\bm x')) = \mu^2 p_b^{-1} = n^{-4+\theta_b} \leq
n^{-4+\theta_{\max}}.
\end{align*}
The number of ordered pairs $(\bm x,\bm x')$ of distinct $r$-APs satisfying
$|\bm x\cap\bm x'|=1$ is $O_r(n^3)$. Hence, 
\begin{align}
\sum_{\substack{\bm x\neq\bm x'\in\cX_{r, n}\\
|\bm x\cap\bm x'|=1}}
\P(\cE_{\bm a}(\bm x)\cap\cE_{\bm a}(\bm x'))
\lesssim_r
n^{-1+\theta_{\max}}  . 
\label{eq:poissononepoint}
\end{align}

\item $|\bm x\cap\bm x'|\geq 2$. As in the proof of Lemma \ref{lm:expectationvariance}, for each fixed $\bm x\in\cX_{r, n}$, there are only $\lesssim_r 1$ choices of $\bm x'\neq\bm x$ sharing at least two points with $\bm x$. Hence, the total number of such ordered pairs is $O_r(n^2)$. If the color requirements of $\bm x$ and $\bm x'$ are incompatible at a common point, then
$\P(\cE_{\bm a}(\bm x)\cap\cE_{\bm a}(\bm x'))=0$.  Otherwise, since $\bm x\neq\bm x'$, the progression $\bm x'$ contains at least one point not belonging to $\bm x$. Conditional on $\cE_{\bm a}(\bm x)$, this new point must receive one of the colors in $S(\bm a)$, and, hence, 
$$
\P(\cE_{\bm a}(\bm x')|\cE_{\bm a}(\bm x))
\leq n^{-\theta_{\min}}.
$$
This implies that 
\begin{align*}
\P(\cE_{\bm a}(\bm x)\cap\cE_{\bm a}(\bm x')) = \P(\cE_{\bm a}(\bm x))
\P(\cE_{\bm a}(\bm x')|\cE_{\bm a}(\bm x)) \leq
n^{-2-\theta_{\min}}  .  
\end{align*}
Therefore, since $\theta_{\min}>0$,  
\begin{align}
\sum_{\substack{\bm x\neq\bm x'\in\cX_{r, n}\\
|\bm x\cap\bm x'|\geq 2}}
\P(\cE_{\bm a}(\bm x)\cap\cE_{\bm a}(\bm x'))
\lesssim_r
n^{-\theta_{\min}}  .  
\label{eq:poissontwopoint}
\end{align}
\end{itemize}  
Combining \eqref{eq:poissononepoint} and \eqref{eq:poissontwopoint}, we obtain
\begin{align}\label{eq:poissonvariance}
B_2  \lesssim_r  n^{-1+\theta_{\max}}+n^{-\theta_{\min}}  .  
\end{align}
Applying \eqref{eq:poissonmain}, \eqref{eq:poissonvariance} in Theorem \ref{thm:poissondependencygraph}, and by the triangle inequality and \eqref{eq:lambdapoisson}, 
\begin{align*}
\mathrm{TV}\!\left( X_{r,n}(\boldsymbol a), \operatorname{Pois}\!\left(\frac{1}{2(r-1)}\right)
\right) & \le \mathrm{TV}\!\left( X_{r,n}(\boldsymbol a), \operatorname{Pois}(\lambda_n)
\right) + \mathrm{TV}\!\left( \operatorname{Pois}(\lambda_n), \operatorname{Pois}\!\left(\frac{1}{2(r-1)}\right) \right)  \nonumber \\ 
&\lesssim_r n^{-1+\theta_{\max}} + n^{-\theta_{\min}} + n^{-1}.
\end{align*}
Since $0<\theta_{\min}\le \theta_{\max}<1$, the $O(\frac{1}{n})$-term can be absorbed in the preceding bound, and the result in Theorem \ref{thm:poisson} follows.  \hfill $\Box$

\section{ Proof of Theorem \ref{thm:intersectiondistribution} }
\label{sec:boundarypoissonpf}

Recall the definition of $\hat s$ from \eqref{eq:location}. For $i\in[n]$, let
\begin{equation*}
\cX_{r,n}^{(i)} \coloneqq \left\{ \bm x =(x,x+d,\ldots,x+(r-1)d) \in \cX_{r,n}: x+(\hat s-1)d=i \right\},
\end{equation*}
be the set of $r$-APs whose $\hat s$-th coordinate is pinned at $i$. Note that for $\bm x\in\cX_{r,n}^{(i)}$, the conditions $i-(\hat s-1)d\geq 1$ and $i+(r-\hat s)d\leq n$ are necessary and sufficient for an $r$-AP with common difference $d$ and $\hat s$-th term equal to $i$. Hence, 
\begin{equation}\label{eq:Xrnlocation}
\vert \cX_{r,n}^{(i)}\vert
=
\min\left\{
\left\lfloor\frac{i-1}{\hat s-1}\right\rfloor,
\left\lfloor\frac{n-i}{r-\hat s}\right\rfloor
\right\}
=
\min\left\{
\frac{i-1}{\hat s-1},
\frac{n-i}{r-\hat s}
\right\}
+O(1),
\end{equation}
where the implicit constant is uniform over $\hat s$ and $i\in[n]$. (Throughout, if one of the two terms in the minimum is undefined, it is omitted from the minimum.) Next, for $\phi_n \sim \mu_{\bm p}$, let 
\begin{equation*}
\cX_{r, n}^{(i)}(\bm a) \coloneq \left\{ \bm x =(x, x+d, \ldots, x+ (r-1) d) \in \cX_{r, n}^{(i)}: \phi_n(x+(t-1) d) = a_t , \text{ for } t \in[r]\setminus\{\hat s\} \right\},
\end{equation*}
be the set of $r$-APs with the $\hat s$-th coordinate pinned at $i \in [n]$ that satisfy all the color requirements of $\bm a$ except, possibly, at the $\hat s$-th coordinate. Define
\begin{equation}\label{eq:Xrnl}
X_{r, n}^{(i)}(\bm a)\coloneq |\cX_{r, n}^{(i)}(\bm a)| = \sum_{\bm x \in  \cX_{r, n}^{(i)}}\mathbf 1\{ \bm x \in  \cX_{r, n}^{(i)}(\bm a) \}.
\end{equation}
Note that our random variable of interest 
\begin{equation}\label{eq:XrnSn}  
X_{r, n} (\bm a)  =  \sum_{i =1}^n X_{r, n}^{(i)}(\bm a) \bm 1\{\phi_n(i)=a_{\hat s}\}  = \sum_{ i \in \cS_n}  X_{r, n}^{(i)}(\bm a)  ,  
\end{equation}
where $\cS_n:= \{ j\in[n] : \phi_n(j)=a_{\hat s}\}$ is the set of sites assigned the rare color.

We now compute the joint (finite-dimensional) falling factorial moments of the collection of random variables $\{X_{r,n}^{(i)}(\bm a)\}_{i\in[n]}$, conditional on $\cS_n$. For any fixed integer $s\geq 0$ and $z\in\R$, denote by $(z)_s=z(z-1)\cdots(z-s+1)$, with the convention $(z)_0=1$. Also, define $\theta_\bullet :=\min_{t\in[r]\setminus\{\hat s\}}\theta_{a_t} \in (0,1)$.

\begin{lemma}
    \label{lm:s}
    Fix an integer $p\geq 1$ and let $B=\{i_1,\ldots,i_p\}\subseteq[n]$ be a collection of distinct indices. Suppose $\bm a$ is an admissible $r$-palette with
    $\sum_{a \in S(\bm a)} w_a \theta_a=2$ and $\theta_{\mathrm{max}}=1$.
    Then, for any collection of non-negative integers $s_1,\ldots,s_p$,
    \begin{equation*}
       \left|  \mathbb E\left[(X_{r,n}^{(i_1)}(\bm a))_{s_1}\cdots
        (X_{r,n}^{(i_p)}(\bm a))_{s_p} | \cS_n = B \right] -
        f_n(i_1)^{s_1}\cdots f_n(i_p)^{s_p} \right|
        \lesssim_{p,r,R} n^{-\theta_\bullet},
    \end{equation*}
where $R=s_1+s_2+\cdots+s_p$ and $f_n : [n] \rightarrow \R_{\geq 0}$ is defined as:  
\begin{equation}
\label{eq:fn}
f_n(i)=\frac{1}{n}\min\left\{ \frac{i-1}{\hat s-1},\frac{n-i}{r-\hat s}\right\}. 
\end{equation}  
\end{lemma}

\begin{proof}
Note that when $R=0$, the result is immediate. Hence, assume $R\geq1$. Conditional on $\cS_n=B$, the colors at the sites of $[n]\setminus B$ are independent. Moreover, for every $b \in [c-1]\backslash \{a_{\hat s}\}$,
\begin{equation}
\mathbb P\bigl(\phi_n(j)=b| \cS_n=B\bigr)
=\frac{p_b}{1-p_{a_{\hat s}}}
=\frac{n^{-\theta_b}}{1-1/n},
\label{eq:phiB}
\end{equation}
for $j\notin B$. Now, for $\bm x=(x,x+d,\ldots,x+(r-1)d)\in\cX_{r,n}^{(i)}$, set
\begin{equation}
D(\bm x)
:=\{x+(t-1)d:t\in[r]\setminus\{\hat s\}\}.
\label{eq:D}
\end{equation}
Thus, $D(\bm x)$ is the set of sites at which
$\bm x\in\cX_{r,n}^{(i)}(\bm a)$ imposes color requirements. If
$D(\bm x)\cap B\neq\varnothing$, then
$$
\mathbb P\bigl(
\bm x\in\cX_{r,n}^{(i)}(\bm a)|\cS_n=B
\bigr)=0,
$$
since $a_t\neq a_{\hat s}$ for every
$t\in[r]\setminus\{\hat s\}$. On the other hand, if
$D(\bm x)\cap B=\varnothing$, then \eqref{eq:phiB} gives
\begin{align}
\mathbb P\bigl(
\bm x\in\cX_{r,n}^{(i)}(\bm a)
|\cS_n=B
\bigr) = \prod_{t\neq\hat s}
\frac{p_{a_t}}{1-p_{a_{\hat s}}} =
\frac{1}{n}
\left(1-\frac{1}{n}\right)^{-(r-1)},
\label{eq:conditionB}
\end{align}
since $\sum_{t\neq\hat s}\theta_{a_t} = \sum_{t=1}^r\theta_{a_t}
-\theta_{a_{\hat s}} =1$.  

Now, expanding the falling factorials gives
\begin{align}
&\mathbb E\left[(X_{r,n}^{(i_1)}(\bm a))_{s_1}\cdots
(X_{r,n}^{(i_p)}(\bm a))_{s_p}|\cS_n=B\right] \nonumber\\
&=
\sum_{\substack{
\bm x^{(1)}_1,\ldots,\bm x^{(1)}_{s_1}\in\cX_{r,n}^{(i_1)}\\
\vdots\\
\bm x^{(p)}_1,\ldots,\bm x^{(p)}_{s_p}\in\cX_{r,n}^{(i_p)}
}}
\P\left(
\bigcap_{u=1}^p\bigcap_{\ell=1}^{s_u}
\left\{
\bm x^{(u)}_\ell\in\cX_{r,n}^{(i_u)}(\bm a)
\right\}
\bigg|\cS_n=B
\right) \nonumber\\
&=:T_1+T_2,
\label{eq:expectationterms}
\end{align}
where, for each $u\in[p]$, the progressions
$\bm x^{(u)}_1,\ldots,\bm x^{(u)}_{s_u}$ are distinct and the  terms $T_1$ and $T_2$ are defined as follows: 
\begin{itemize}

\item $T_1$ is the sum over collections for which the sets $\{D(\bm x^{(u)}_\ell):\ell\in[s_u], u\in[p]\}$ are pairwise disjoint and all avoid $B$.

\item $T_2$ is the sum over all remaining collections.

\end{itemize}
For notational convenience, let $\bm s=(s_1,\ldots,s_p)$ and denote by
$\cA_{\bm s}(B)$ the range of the summation in
\eqref{eq:expectationterms}. Thus,
\begin{align*}
|\cA_{\bm s}(B)| = \prod_{u=1}^p \left(|\cX_{r,n}^{(i_u)}|\right)_{s_u} = n^R\prod_{u=1}^p f_n(i_u)^{s_u} +O_{p,r,R}(n^{R-1}),
\end{align*}
by \eqref{eq:Xrnlocation} and \eqref{eq:fn}.

We first estimate the number of terms in $T_1$. The number of
$R$-tuples in $\cA_{\bm s}(B)$ for which there exists a progression
$\bm x$ satisfying $D(\bm x)\cap B\neq\varnothing$ is $O_{p,r,R}(n^{R-1})$. This is because, first choose the position of $\bm x$
among the $R$ progressions, a point of $B$, and a coordinate
$t\in[r]\setminus\{\hat s\}$. Since the $\hat s$-th coordinate
of $\bm x$ is fixed, these choices determine its common difference,
if such an integer common difference exists. The remaining $R-1$
progressions have $O(n^{R-1})$ choices. Similarly, the number of $R$-tuples containing a pair $(\bm x,\bm x')$ such that $D(\bm x)\cap D(\bm x')\neq\varnothing
$ is $O_{p,r,R}(n^{R-1})$. Indeed, choose the two positions of
$\bm x,\bm x'$ among the $R$ progressions and coordinates
$t,t'\in[r]\setminus\{\hat s\}$ at which they intersect.
After choosing the common difference of $\bm x$ in $O(n)$ ways,
the common difference of $\bm x'$ is determined by the equality
of the two selected coordinates. The remaining $R-2$ progressions
have $O(n^{R-2})$ choices. Therefore, the number of terms contributing to $T_1$ is
\begin{align}
n^R\prod_{u=1}^p f_n(i_u)^{s_u}
+O_{p,r,R}(n^{R-1}).
\label{eq:disjoint}
\end{align} 
For every term in $T_1$, the corresponding events are independent
conditional on $\cS_n=B$, since they depend on disjoint subsets of
$[n]\setminus B$. Hence, by \eqref{eq:conditionB},
\begin{align*}
\P\left(
\bigcap_{u=1}^p\bigcap_{\ell=1}^{s_u}
\left\{
\bm x^{(u)}_\ell\in\cX_{r,n}^{(i_u)}(\bm a)
\right\}
\bigg|\cS_n=B
\right)  =
\frac{1}{n^R}
\left(1-\frac{1}{n}\right)^{-R(r-1)}
=
\frac{1}{n^R}
+O_{r,R}\left(\frac{1}{n^{R+1}}\right).
\end{align*}
Combining this with \eqref{eq:disjoint} gives
\begin{align}
T_1
=
\prod_{u=1}^p f_n(i_u)^{s_u}
+O_{p,r,R}(n^{-1}).
\label{eq:factorialdisjoint}
\end{align}

We now analyze $T_2$. Any tuple containing a progression $\bm x$
with $D(\bm x)\cap B\neq\varnothing$ has conditional probability zero.
Thus, it suffices to consider tuples in $T_2$ for which all the sets
$D(\bm x)$ avoid $B$. Now, for $\bm X\in\cA_{\bm s}(B)$, define its intersection graph
$G_{\bm X}$ as follows. The vertices of $G_{\bm X}$ are the $R$
progressions in $\bm X$, and two distinct vertices $\bm x,\bm x'$
are adjacent if $D(\bm x)\cap D(\bm x')\neq\varnothing$. Denote by $\nu(G_{\bm X})$ the number of connected components of
$G_{\bm X}$ and by $\nu_1(G_{\bm X})$ the number of connected
components containing at least one edge. For every nonzero term in
$T_2$, we have $\nu_1(G_{\bm X})\geq1$. Now, fix a possible labeled intersection graph $G$ on the $R$ progressions, and let $F$ be a spanning forest of $G$. For every edge of
$F$, connecting say $\bm x$ and $\bm x'$, choose a pair of coordinates
$t,t'\in[r]\setminus\{\hat s\}$ witnessing an intersection between
$D(\bm x)$ and $D(\bm x')$. There are only $O_r(1)$ choices for these
coordinates. Choose one root in each connected component of $F$. The common
difference of each root progression can be chosen in $O(n)$ ways.
Once the common difference of a parent progression is fixed, the
chosen intersection relation along an edge of $F$ uniquely determines
the common difference of the child progression, if a valid integer
common difference exists. Therefore, for every fixed graph $G$ and
fixed choices of witnessing coordinates, the number of corresponding
$R$-tuples is 
\begin{align}\label{eq:graphF}
O_{p,r,R}(n^{\nu(G)}). 
\end{align}
Since $R$ is fixed, there are only finitely many possible graphs and
choices of witnessing coordinates. Hence, it remains to bound the probability associated with a fixed tuple $\bm X$. To  this end, let $F_0$ be a connected component of $G_{\bm X}$. Now, we consider the following two cases: 

\begin{itemize} 

\item If $F_0$
consists of a single progression $\bm x$, then by
\eqref{eq:conditionB},
\begin{equation}
\P(\bm x\in\cX_{r,n}^{(i)}(\bm a) |\cS_n=B ) \lesssim_r n^{-1}.
\label{eq:vertex}
\end{equation}

\item Next, suppose $F_0$ contains at least one edge. Choose two adjacent
distinct progressions $\bm x,\bm x'$ in $F_0$. If their color
requirements are incompatible at any point of
$D(\bm x)\cap D(\bm x')$, then
$$
\P\left(
\bm x,\bm x'
\text{ both satisfy their color requirements}
|\cS_n=B
\right)=0.
$$
Otherwise, the color requirements are compatible at every common
point. Since $\bm x\neq\bm x'$, the sets $D(\bm x)$ and $D(\bm x')$ cannot
be identical: the $r-1\geq2$ coordinates in $D(\bm x)$ determine the
underlying $r$-AP uniquely. Hence, there exists at least one point
$j\in D(\bm x')\setminus D(\bm x).$ Conditional on the event that $\bm x$ satisfies its color requirements, the color at $j$ is still independent of all colors in $D(\bm x)$. Therefore, by \eqref{eq:phiB},
\begin{align*}
\P (\bm x'\in\cX_{r,n}^{(i_v)}(\bm a) | \bm x\in\cX_{r,n}^{(i_u)}(\bm a), 
\cS_n=B) \lesssim_r n^{-\theta_\bullet}.
\end{align*}
Together with \eqref{eq:vertex}, this gives
\begin{align}
&\P\left(
\bm x\in\cX_{r,n}^{(i_u)}(\bm a), 
\bm x'\in\cX_{r,n}^{(i_v)}(\bm a)
|\cS_n=B
\right)
\lesssim_r
n^{-1-\theta_\bullet}.
\label{eq:edgecomponent}
\end{align}
Since the event that all progressions in $F_0$ satisfy their color
requirements is contained in the event appearing in
\eqref{eq:edgecomponent}, the same upper bound holds for the entire
component $F_0$.
\end{itemize}
Finally, distinct connected components of $G_{\bm X}$ involve disjoint
sets of constrained sites. Hence, conditional on $\cS_n=B$, the
corresponding events are independent. It follows from
\eqref{eq:vertex} and \eqref{eq:edgecomponent} that
\begin{align}
\P\left(
\bigcap_{u=1}^p\bigcap_{\ell=1}^{s_u}
\left\{
\bm x^{(u)}_\ell\in\cX_{r,n}^{(i_u)}(\bm a)
\right\} \bigg|\cS_n=B \right) & \lesssim_{p,r,R}
n^{-\nu(G_{\bm X})-\nu_1(G_{\bm X})\theta_\bullet}.
\label{eq:C}
\end{align}
Combining \eqref{eq:graphF} and \eqref{eq:C}, the contribution of all tuples with
any fixed intersection pattern is at most
\begin{align*}
n^{\nu(G_{\bm X})}
n^{-\nu(G_{\bm X})-\nu_1(G_{\bm X})\theta_\bullet} = n^{-\nu_1(G_{\bm X})\theta_\bullet} \leq n^{-\theta_\bullet},
\end{align*}  
since $\nu_1(G_{\bm X})\geq1$. There are $O_{p,r,R}(1)$ choices of  the  possible intersection patterns, 
\begin{equation}
T_2
\lesssim_{p,r,R}
n^{-\theta_\bullet}.
\label{eq:cross}
\end{equation}
Combining \eqref{eq:factorialdisjoint} and \eqref{eq:cross}, the result in Lemma \ref{lm:s} follows. 
\end{proof}

With the above lemma, we can now complete the proof of Theorem \ref{thm:intersectiondistribution}. Let 
\begin{equation*}
\bm X_{r,n}^{\cS_n}(\bm a) := \bigl(X_{r,n}^{(i)}(\bm a)\bigr)_{i\in \cS_n}.
\end{equation*}
Conditional on $\cS_n$, let
$\bm Y_n' :=(Y_i')_{i\in \cS_n}$, where $Y_i' \mid \cS_n\sim\mathrm{Pois}(f_n(i))$,
independently for $i\in \cS_n$. Also, conditional on $\cS_n$, let 
$\bm Y_n :=(Y_i)_{i\in \cS_n}$, where $Y_i\mid \cS_n
\sim \mathrm{Pois}\left(\lambda_{\bm a}\left(\frac{i}{n}\right)\right)$,
independently for $i\in \cS_n$. Using the bound $\mathrm{TV}(\mathrm{Pois}(\lambda),\mathrm{Pois}(\lambda')) \leq |\lambda-\lambda'| $ and the tensorization property of Total Variation distance, conditional on $\cS_n$, 
\begin{align*}
\mathrm{TV}\left( (\bm Y_n'\mid\cS_n), (\bm Y_n\mid\cS_n)
\right) \leq \sum_{i\in\cS_n} \left|
f_n(i)-\lambda_{\bm a}\left(\frac{i}{n}\right)
\right| \lesssim_r \frac{|\cS_n|}{n}.
\end{align*}
Since $\E|\cS_n|=1$, it follows that, as $n \rightarrow \infty$,  
\begin{align}
\mathrm{TV}\left( (\cS_n,\bm Y_n'), (\cS_n,\bm Y_n)
\right) \rightarrow 0.
\label{eq:YZ}
\end{align}

Next, we control the Total Variation distance between $\bm X_{r,n}^{\cS_n}(\bm a)$ and $\bm Y_n'$. 

\begin{lemma} For $\bm X_{r,n}^{\cS_n}(\bm a)$ and $\bm Y_n'$ as defined above, 
\begin{align}
\mathrm{TV} ( (\cS_n, \bm X_{r,n}^{\cS_n}(\bm a)),  (\cS_n,\bm Y_n')  )  \rightarrow 0  . 
\label{eq:SY}
\end{align}
\end{lemma}

\begin{proof} For fixed $M \geq 1$, define 
\begin{equation*}
  \Delta_{n,M}  :=  \sup_{ \substack{ \bm s = (s_k)_{k \in \cS_n} \in \N_0^{\cS_n}  \\ 
  \sum_{k \in \cS_n} s_k   \leq M } }  \left|  \mathbb E\left[\prod_{i \in \cS_n} (X_{r,n}^{(i)}(\bm a))_{s_i}  | \cS_n \right] -
        \prod_{i \in \cS_n}f_n(i)^{s_i} \right|  .  
    \end{equation*}  
   Also, define 
   $$\Lambda_n: = \sum_{i\in\cS_n}f_n(i).$$ Note that $|\cS_n|\sim\mathrm{Bin}(n, \frac{1}{n})$, hence, the sequence $\{|\cS_n|\}_{n\geq1}$ is tight.  Also, since $\Lambda_n \leq |\cS_n|$, the sequence $\{ \Lambda_n \}_{n\geq1}$ is tight. Hence, for $\varepsilon >0$, by tightness, there exist $L \geq 1$ and $K<\infty$ such that $\mathbb P(| \cS_n|>L)< \varepsilon$ and  $\mathbb P(\Lambda_n>K) < \varepsilon$, for all sufficiently large $n$. Apply Lemma~\ref{facMomentTVbd} with this $L,K$, and $\varepsilon$. This gives $M\geq 1$ and $\delta>0$ such that, on the event
\[
\{| \cS_n|\leq L\}\cap\{\Lambda_n\leq K\}
\cap\{\Delta_{n,M}\leq\delta\},
\]
we have $\mathrm{TV}( \bm X_{r,n}^{\cS_n}(\bm a) \mid  \cS_n,  \bm  Y_n' \mid  \cS_n) \leq\varepsilon$. Hence, 
\begin{align*}
\mathbb P\left( \mathrm{TV}\left( \bm X_{r,n}^{\cS_n}(\bm a) \mid  \cS_n,
  \bm  Y_n' \mid  \cS_n  \right)>\varepsilon
\right) \leq \mathbb P(| \cS_n|>L) +\mathbb P(\Lambda_n>K) +\mathbb P(\Delta_{n,M}>\delta).
\end{align*}
Letting $n\to\infty$ followed by $\varepsilon \rightarrow 0$ and applying Lemma \ref{lm:s} gives, 
\begin{equation}
\mathrm{TV}\left( \bm X_{r,n}^{\cS_n}(\bm a) \mid \cS_n, \bm  Y_n' \mid \cS_n \right)
\xrightarrow{P}0.
\label{eq:YZSn}
\end{equation}
To now derive the unconditional statement,  suppose $A$ is a measurable set in the common sample space. Then
\begin{align*}
\left| \mathbb P\bigl(( \cS_n,\bm X_{r,n}^{\cS_n}(\bm a)) \in A\bigr) - \mathbb P\bigl(( \cS_n,\bm  Y_n')\in A\bigr) \right| & = \left| \E \left[ \mathbb P\bigl(( \cS_n,\bm X_{r,n}^{\cS_n}(\bm a)) \in A | \cS_n \bigr) - \mathbb P\bigl(( \cS_n,\bm  Y_n')\in A | \cS_n \bigr) \right] \right| \nonumber \\ 
& \leq \mathbb E\left[ \mathrm{TV}\left(   \bm X_{r,n}^{\cS_n}(\bm a) \mid  \cS_n ,   \bm  Y_n'\mid  \cS_n \right) \right].
\end{align*}
Note that taking the supremum over $A$ in the LHS gives the Total Variation distance. Further, the RHS converges to zero by \eqref{eq:YZSn} and since the conditional Total Variation distance is bounded by $1$. This proves \eqref{eq:SY}.  
\end{proof}

Combining \eqref{eq:YZ} and \eqref{eq:SY} gives, $\mathrm{TV}( (\cS_n,\bm X_{r,n}^{\cS_n}(\bm a) ), (\cS_n, \bm Y_n)) \rightarrow 0$. Note that the Total Variation distance cannot increase under a measurable map. Hence, applying the summation map and recalling \eqref{eq:XrnSn}, we obtain 
 \begin{align}\label{eq:XrnTV}
 \mathrm{TV}\left( X_{r,n}(\bm a) , \sum_{i\in \cS_n}Y_{i} \right) \rightarrow 0. 
 \end{align} 
 Now, define the random finite measure $ \Xi_n = \sum_{i\in\cS_n}Y_i \delta_{\frac{i}{n}}$. By Lemma~\ref{lem:comPoisson},
$$\Xi_n \xrightarrow{D} \Xi:=\sum_{x\in\mathcal P}  N_x\delta_x,$$
where $\mathcal P$ and $\{ N_x\}_{x\in\mathcal P}$ are as defined in Theorem \ref{thm:intersectiondistribution}. Then, by the continuous mapping theorem, $\sum_{i\in\cS_n}Y_i \stackrel{D} \rightarrow \sum_{x\in\mathcal P} N_x$.
 Together with \eqref{eq:XrnTV}, this proves the result in Theorem \ref{thm:intersectiondistribution}.  \hfill $\Box$

\section{Proof of Theorem \ref{thm:cubedistribution}}
\label{sec:cubedistributionpf}

To prove Theorem \ref{thm:cubedistribution}, we first compute the conditional
mean and variance of the random variables $X_{r,n}^{(i)}(\bm a)$ (recall \eqref{eq:Xrnl}).

\begin{lemma}
\label{lem:pinned-concentration} 
Suppose the assumptions of Theorem \ref{thm:cubedistribution} hold. Let $B\subseteq[n]$ be a finite set and suppose $i\in B$. Then
\begin{align}
\E\left[
X_{r,n}^{(i)}(\bm a)
 | 
\cS_n=B
\right]
&=
n^{2 - \sum_{a\in S(\bm a)}w_a\theta_a}f_n(i)
+
O_{r,|B|}\left(n^{1 - \sum_{a\in S(\bm a)}w_a\theta_a}\right),
\label{eq:pinned-mean}
\end{align}
and
\begin{equation}
\Var\left[ X_{r,n}^{(i)}(\bm a) | \cS_n=B \right] \lesssim_r n^{2 - \sum_{a\in S(\bm a)}w_a\theta_a}.
\label{eq:pinned-var}
\end{equation}
Consequently, for $i\in B$, 
$$\frac{X_{r,n}^{(i)}(\bm a)}{n^{2 - \sum_{a\in S(\bm a)}w_a\theta_a}} - f_n(i)
\xrightarrow{P}0,$$ 
conditional on $\cS_n=B$.
\end{lemma}

\begin{proof}
Recall that, for $\bm x=(x,x+d,\ldots,x+(r-1)d)\in\cX_{r,n}^{(i)}$, we write $D(\bm x) =
\{x+(t-1)d:t\in[r]\setminus\{\hat s\}\}$. Hence, if $D(\bm x)\cap B=\varnothing$, then by arguments similar to \eqref{eq:conditionB},
\begin{align}
\P\left(\bm x\in\cX_{r,n}^{(i)}(\bm a)  |  \cS_n=B \right) = n^{1 - \sum_{a\in S(\bm a)}w_a\theta_a} \left(1-\frac1n\right)^{-(r-1)}.
\label{eq:xrnconditional}
\end{align}
On the other hand, if $D(\bm x)\cap B\neq\varnothing$, then the conditional probability is zero. Further, for every $j\in B$ and every $t\in[r]\setminus\{\hat s\}$, there is at most one progression $\bm x\in\cX_{r,n}^{(i)}$ such that $x+(t-1)d=j$, because its $\hat s$-th coordinate is already fixed at $i$. Therefore,
\begin{equation}
\left|
\left\{
\bm x\in\cX_{r,n}^{(i)}
:
D(\bm x)\cap B\neq\varnothing
\right\}
\right|
\lesssim_r |B|.
\label{eq:bad-pinned}
\end{equation}
Using \eqref{eq:xrnconditional}, \eqref{eq:bad-pinned}, and $|\cX_{r,n}^{(i)}| = nf_n(i)+O_r(1)$ (which follows from \eqref{eq:Xrnlocation} and \eqref{eq:fn}), we obtain
\begin{align*}
\E\left[ X_{r,n}^{(i)}(\bm a) |  \cS_n=B \right] & = \left( |\cX_{r,n}^{(i)}| + O_r(|B|) \right) n^{1 - \sum_{a\in S(\bm a)}w_a\theta_a} \left(1-\frac1n\right)^{-(r-1)}  \nonumber  \\  
  & = n^{2 - \sum_{a\in S(\bm a)}w_a\theta_a}f_n(i) + O_{r,|B|}\left(n^{1 - \sum_{a\in S(\bm a)}w_a\theta_a}\right),
\end{align*}
which proves \eqref{eq:pinned-mean}.  

We now bound the variance. Recalling \eqref{eq:Xrnl}, we decompose the $\Var[ X_{r,n}^{(i)}(\bm a) ]$ into variance and covariance terms. Note that if $D(\bm x)\cap D(\bm x')=\varnothing$, then $ \bm 1\{\bm x\in\cX_{r,n}^{(i)}(\bm a)  \}  $ and $ \bm 1\{\bm x' \in\cX_{r,n}^{(i)}(\bm a)  \}  $ are independent conditional on $\cS_n=B$. Moreover, for
each fixed $\bm x\in\cX_{r,n}^{(i)}$, there are only $O_r(1)$ choices
of $\bm x'\in\cX_{r,n}^{(i)}$ such that $D(\bm x)\cap D(\bm x')\neq\varnothing$. This is because, if the common point occurs at position $t$ of $\bm x$ and position $t'$ of $\bm x'$, where
$t,t'\in[r]\setminus\{\hat s\}$, then $(t-\hat s)d=(t'-\hat s)d'$, and hence $d'$ is uniquely determined by $d,t$, and $t'$. For such a pair, using \eqref{eq:xrnconditional},
\begin{align*}
& \left| \Cov[ \bm 1\{\bm x\in\cX_{r,n}^{(i)}(\bm a)  \} , \bm 1\{\bm x' \in\cX_{r,n}^{(i)}(\bm a)  \} \mid\cS_n=B ] \right|  \nonumber \\ 
&\leq | \P( \bm x , \bm x' \in \cX_{r,n}^{(i)}(\bm a)  \mid\cS_n=B  )  +  \P( \bm x  \in\cX_{r,n}^{(i)}(\bm a) \mid\cS_n=B) \P( \bm x' \in\cX_{r,n}^{(i)}(\bm a) \mid\cS_n=B  )  \\  
&\lesssim_r
n^{1 - \sum_{a\in S(\bm a)}w_a\theta_a}  .  
\end{align*}
Since $|\cX_{r,n}^{(i)}|=O_r(n)$, this proves \eqref{eq:pinned-var}.  

Hence, 
\begin{align*}
\Var\left[ \frac{X_{r,n}^{(i)}(\bm a)}
{n^{2 - \sum_{a\in S(\bm a)}w_a\theta_a}} |  \cS_n=B
\right] \lesssim_r n^{ - (2 - \sum_{a\in S(\bm a)}w_a\theta_a) }  \ll  1  ,  
\end{align*}
since $\sum_{a\in S(\bm a)}w_a\theta_a < 2$. Also, by \eqref{eq:pinned-mean},
\[
\frac{
\E[X_{r,n}^{(i)}(\bm a)\mid\cS_n=B]
}
{n^{2 - \sum_{a\in S(\bm a)}w_a\theta_a}}
=
f_n(i)+O_{r,|B|}\left(\frac1n\right).
\]
The result now follows from Chebyshev's inequality.
\end{proof}

We can now proceed to prove Theorem \ref{thm:cubedistribution}. Recall from \eqref{eq:XrnSn} that $X_{r,n}(\bm a) = \sum_{i\in\cS_n}X_{r,n}^{(i)}(\bm a)$. Since $|\cS_n| \sim \mathrm{Bin}\left(n,\frac1n\right)$, the sequence $\{|\cS_n|\}_{n\geq1}$ is tight. We first show that
\begin{equation}
\frac{X_{r,n}(\bm a)}
{n^{2 - \sum_{a\in S(\bm a)}w_a\theta_a}} - \sum_{i\in\cS_n}f_n(i)
\xrightarrow{P}0.
\label{eq:Xrnintersection}
\end{equation}
By the tightness of $|\cS_n|$, for $\varepsilon > 0$, there exists $L\geq1$ such that $\P(|\cS_n|>L)< \varepsilon$, for all sufficiently large $n$. Conditional on $\cS_n=B$ with $|B|\leq L$, define
$$Z_{n}^{(i)} := \frac{X_{r,n}^{(i)}(\bm a)}
{n^{2 - \sum_{a\in S(\bm a)}w_a\theta_a}} - f_n(i), $$
for $i\in B$. By Lemma \ref{lem:pinned-concentration} and the Cauchy-Schwarz inequality,
\begin{align*}
\E\left[ \left( \sum_{i\in B}Z_{n}^{(i)} \right)^2 \big| \cS_n=B \right] &\leq |B|\sum_{i\in B}
\E\left[\left(Z_{n}^{(i)}\right)^2 \big| \cS_n=B \right] \\ 
&\lesssim_{r,L}
n^{ - (2 - \sum_{a\in S(\bm a)}w_a\theta_a) } + \frac1{n^2}  .  
\end{align*}
Since $\sum_{a\in S(\bm a)}w_a\theta_a < 2$, the RHS tends to zero uniformly over all $B$ with $|B|\leq L$. Therefore, for any $\eta > 0$, as $n \rightarrow \infty$, 
$$\P\left( \left| \frac{X_{r,n}(\bm a)} {n^{2 - \sum_{a\in S(\bm a)}w_a\theta_a}} - \sum_{i\in\cS_n}f_n(i)
\right|> \eta,  |\cS_n|\leq L \right) \rightarrow 0.$$
Together with $\P(|\cS_n|>L)< \varepsilon$, and taking limits as $n \rightarrow \infty$ and then $\varepsilon \rightarrow 0$, gives the result in \eqref{eq:Xrnintersection}.  

It remains to determine the limiting distribution of $\sum_{i\in\cS_n}f_n(i)$.
Note that, by the definitions of $f_n$ (recall \eqref{eq:fn}) and $\lambda_{\bm a}$ (as in \eqref{eq:lambda}),  
\begin{align}
\left| \sum_{i\in\cS_n}f_n(i)  -  \sum_{i\in\cS_n} \lambda_{\bm a}\left(\frac{i}{n}\right)
\right| &\lesssim_r \frac{|\cS_n|}{n} \xrightarrow{P}0,
\label{eq:f-lambda-sum}
\end{align}
again using the tightness of $|\cS_n|$. Now,  define the point process of rare-color locations by
\[
\Pi_n
:=
\sum_{i\in\cS_n}\delta_{\frac{i}{n}}
=
\sum_{i=1}^n
\bm 1\{\phi_n(i)=a_{\hat s}\}\delta_{\frac{i}{n}}.
\]
Since the indicators $\{ \bm 1\{\phi_n(i)=a_{\hat s}\}\}_{i\in[n]}$, are independent $\operatorname{Ber}(\frac{1}{n})$ random variables, the
standard Poisson point process limit gives
\begin{equation}  
\Pi_n
\xrightarrow{D}
\Pi
:=
\sum_{x\in\cP}\delta_x,
\label{eq:intersectionprocess}
\end{equation}
where $\cP$ is a Poisson point process of rate $1$ on $[0,1]$. Hence, by the continuous
mapping theorem and \eqref{eq:intersectionprocess},
\begin{align}
\sum_{i\in\cS_n} \lambda_{\bm a}\left(\frac{i}{n}\right) = \int_{[0,1]} \lambda_{\bm a}(x) \mathrm d\Pi_n(x) \xrightarrow{D} \int_{[0,1]} \lambda_{\bm a}(x) \mathrm d\Pi(x) = \sum_{x\in\cP}\lambda_{\bm a}(x).
\label{eq:lambdaprocess}
\end{align}
Combining \eqref{eq:f-lambda-sum} and \eqref{eq:lambdaprocess} gives
\begin{equation*}
\sum_{i\in\cS_n}f_n(i)
\xrightarrow{D}
\sum_{x\in\cP}\lambda_{\bm a}(x).
\end{equation*}
This, together with \eqref{eq:Xrnintersection} proves the result in Theorem \ref{thm:cubedistribution}.  \hfill $\Box$

\small 

\subsection*{Acknowledgements} BBB was supported by NSF CAREER grant DMS 2046393  and the National University of Singapore's PESS and Provost's Chair funds.

\subsection*{AI disclosure} ChatGPT 5.6 Sol was used to verify some of the mathematical arguments, refine grammar and typography, and prepare the figures presented in this paper.  The research themes and results were derived by the authors, who assume full responsibility for all content.

\bibliographystyle{plainnat} 
\bibliography{bibliography}

\normalsize

\appendix

\section{ Technical Lemmas }
\label{sec:appendix}

In this section we collect a few technical lemmas about Poisson convergence. Throughout, given $\bm \alpha=(\alpha_x)_{x\in S}\in\mathbb N_0^S$, for a finite set $S$, we denote $|\bm \alpha|:=\sum_{x\in S}\alpha_x$.

\begin{lemma}
\label{facMomentTVbd}
Fix an integer $L\geq 1$ and $K<\infty$. For every $\varepsilon>0$, there exist an integer $M\geq1$ and a number $\delta>0$ such that the following holds. Suppose $S$ is a finite set, with $|S|\leq L$, and $\bm\lambda=(\lambda_x)_{x\in S}$ is such that $\lambda_x\geq0$ and $\sum_{x\in S}\lambda_x\leq K$. Let $\nu_{\bm\lambda} = \bigotimes_{x\in S}\operatorname{Pois}(\lambda_x)$
be the law of independent Poisson random variables with means $(\lambda_x)_{x\in S}$. Suppose $\mu$ is a probability measure on
$\mathbb N_0^S$ such that
\begin{equation}
\sup_{\substack{\bm\alpha\in\mathbb N_0^S\\1\leq|\bm\alpha|\leq M}}
\left|
 \int \prod_{x\in S}(z_x)_{\alpha_x} \mathrm d\mu(\bm z)
 -\prod_{x\in S}\lambda_x^{\alpha_x}
\right|
\leq\delta.
\label{eq:sz}
\end{equation}
Then $\mathrm{TV}(\mu,\nu_{\bm\lambda})\leq\varepsilon$.

\end{lemma}

\begin{proof}
Suppose the assertion is false. Then there exists $\varepsilon_0>0$
such that, for every integer $M\geq1$ and every $\delta>0$, there exist
a finite set $S$, a vector $\bm\lambda$, and a probability measure
$\mu$ satisfying the assumptions of the lemma and \eqref{eq:sz}, but $\mathrm{TV}(\mu,\nu_{\bm\lambda})>\varepsilon_0$. In particular, for every integer $t\geq1$, there exist finite sets
$S_t$ with $|S_t|\leq L$, vectors
$\bm\lambda^{(t)}=(\lambda_x^{(t)})_{x\in S_t}$, satisfying $\lambda_x^{(t)}\geq0$ and 
$\sum_{x\in S_t}\lambda_x^{(t)}\leq K$, and probability measures $\mu_t$ on $\mathbb N_0^{S_t}$ such that
\begin{equation}
\sup_{\substack{\bm\alpha\in\mathbb N_0^{S_t}\\1\leq|\bm\alpha|\leq t}}
\left|
 \int \prod_{x\in S_t}(z_x)_{\alpha_x} \mathrm d\mu_t(\bm z)
 -\prod_{x\in S_t}(\lambda_x^{(t)})^{\alpha_x}
\right|
\leq\frac{1}{t},
\label{eq:fac-moment-contradiction-sequence}
\end{equation}
but
\begin{equation}
\mathrm{TV}(\mu_t,\nu_{\bm\lambda^{(t)}})
>\varepsilon_0.
\label{eq:TV-separated}
\end{equation}
Considering a subsequence, we may assume that $|S_t|=\ell$ is fixed.
The case $\ell=0$ is trivial, so assume $\ell\geq1$. Relabeling the
elements of $S_t$, we identify every $S_t$ with $[\ell]$. Considering another subsequence, we may also assume that $\bm\lambda^{(t)} \rightarrow\bm\lambda$ in $[0,K]^\ell$, for some $\bm\lambda=(\lambda_1,\ldots,\lambda_\ell)$ satisfying
$\sum_{j=1}^\ell\lambda_j\leq K$. Let $T(\bm z):=\sum_{j=1}^\ell z_j$, for $\bm z = (z_1, z_2, \ldots, z_\ell)$. Taking $\bm\alpha$ to be each of the $\ell$ coordinate unit vectors in  \eqref{eq:fac-moment-contradiction-sequence}, we obtain, for all sufficiently large $t$,
\[
\int T(\bm z) \mathrm d\mu_t(\bm z)
\leq
\sum_{j=1}^\ell\lambda_j^{(t)}
+\frac{\ell}{t}
\leq K+\ell.
\]
Hence, $\{\mu_t\}_{t\geq1}$ is tight on the countable space
$\mathbb N_0^\ell$. This implies that, along a subsequence, we can assume that
\begin{equation}
\mu_t\Rightarrow\mu
\label{eq:mu-weak-limit}
\end{equation}
for some probability measure $\mu$ on $\mathbb N_0^\ell$.

We next identify the mixed factorial moments of $\mu$. Fix
$\bm\alpha=(\alpha_1,\ldots,\alpha_\ell)\in\mathbb N_0^\ell$. If $|\bm\alpha|=0$, there is nothing to prove. Choose an integer $\gamma > |\bm\alpha| $.  For every fixed $s\leq\gamma$, by the multinomial expansion, 
\begin{equation}
(T(\bm z))_s = \sum_{\substack{\bm\beta=(\beta_1,\ldots,\beta_\ell)\in\mathbb N_0^\ell\\
\beta_1+\cdots+\beta_\ell=s}}
\frac{s!}{\beta_1!\cdots\beta_\ell!}
\prod_{j=1}^\ell(z_j)_{\beta_j}  .  
\label{eq:factorial-multinomial}
\end{equation} 
Then \eqref{eq:fac-moment-contradiction-sequence} implies, for all
sufficiently large $t$,
\begin{align}
\int (T(\bm z))_s \mathrm d\mu_t(\bm z)
&\leq
\sum_{\substack{\bm\beta\in\mathbb N_0^\ell\\
\beta_1+\cdots+\beta_\ell=s}}
\frac{s!}{\beta_1!\cdots\beta_\ell!}
\prod_{j=1}^\ell(\lambda_j^{(t)})^{\beta_j} +
\frac{1}{t}
\sum_{\substack{\bm\beta\in\mathbb N_0^\ell\\
\beta_1+\cdots+\beta_\ell=s}}
\frac{s!}{\beta_1!\cdots\beta_\ell!}   \nonumber  \\
&=
\left(\sum_{j=1}^\ell\lambda_j^{(t)}\right)^s
+\frac{\ell^s}{t}
\leq K^s+\ell^s.
\label{eq:factorial-T-bound}
\end{align}
Note that 
\begin{equation*}
T^\gamma = \sum_{s=0}^\gamma \left\{\begin{matrix}\gamma\\s\end{matrix}\right\}(T)_s  ,  
\end{equation*}
where the coefficients are Stirling numbers of the second kind. Then \eqref{eq:factorial-T-bound} gives $$\sup_{t\geq\gamma} \int T(\bm z)^\gamma \mathrm d\mu_t(\bm z) <\infty.$$ Consequently, for $A>0$,
\begin{align*}
\sup_{t\geq\gamma}
\int T(\bm z)^{|\bm \alpha|} \mathbf 1\{T(\bm z)^{|\bm \alpha|}>A\} \mathrm d\mu_t(\bm z)  \leq A^{-\frac{\gamma-^{|\bm \alpha|} }{ ^{|\bm \alpha|} }}
\sup_{t\geq\gamma} \int T(\bm z)^\gamma \mathrm d\mu_t(\bm z) \rightarrow 0  ,  
\end{align*}
as $A\to\infty$. Thus $\{T^{|\bm \alpha|}\}_{t\geq\gamma}$ is uniformly
integrable. Moreover, $$0\leq \prod_{j=1}^\ell(z_j)_{\alpha_j} \leq \prod_{j=1}^\ell z_j^{\alpha_j} \leq
T(\bm z)^{|\bm \alpha|}.$$ Therefore, $\{\prod_{j=1}^\ell(z_j)_{\alpha_j} \}_{t\geq\gamma}$ is uniformly integrable as well. Combining \eqref{eq:fac-moment-contradiction-sequence}, \eqref{eq:mu-weak-limit}, and
$\bm\lambda^{(t)}\to\bm\lambda$, gives, 
\begin{equation}
\int
\prod_{j=1}^\ell(z_j)_{\alpha_j}
 \mathrm d\mu(\bm z)
=
\prod_{j=1}^\ell\lambda_j^{\alpha_j}.
\label{eq:limit-factorial-moments}
\end{equation}

It remains to show that these factorial moments determine the
distribution. Let $\bm Z=(Z_1,\ldots,Z_\ell)$ have law $\mu$, and set $T:=\sum_{j=1}^\ell Z_j$
and $\Lambda:=\sum_{j=1}^\ell\lambda_j$. By \eqref{eq:factorial-multinomial} and \eqref{eq:limit-factorial-moments},
\begin{align*}
\E[(T)_s] = \sum_{\substack{\bm\alpha=(\alpha_1,\ldots,\alpha_\ell)\in\mathbb N_0^\ell\\
\alpha_1+\cdots+\alpha_\ell=s}}
\frac{s!}{\alpha_1!\cdots\alpha_\ell!}
\E\left[
\prod_{j=1}^\ell(Z_j)_{\alpha_j}
\right] &=
\sum_{\substack{\bm\alpha\in\mathbb N_0^\ell\\
\alpha_1+\cdots+\alpha_\ell=s}}
\frac{s!}{\alpha_1!\cdots\alpha_\ell!}
\prod_{j=1}^\ell\lambda_j^{\alpha_j}\\
&=
\left(\sum_{j=1}^\ell\lambda_j\right)^s
=
\Lambda^s  ,  
\end{align*}
for every integer $s\geq0$. Since $2^T = \sum_{s=0}^\infty\frac{1}{s!} (T)_s $, monotone convergence gives,  
\begin{equation}
\E[2^T]
=
\sum_{s=0}^\infty\frac{\E[(T)_s]}{s!}
=
e^\Lambda
<\infty.
\label{eq:exponential-moment}
\end{equation}  
Now, for $(s_1,\ldots,s_\ell)\in[0,1]^\ell$, the factorial expansion
may be integrated term by term, giving
\begin{align}
\E\left[\prod_{j=1}^\ell s_j^{Z_j}\right]
&=
\sum_{\alpha_1,\ldots,\alpha_\ell\geq0}
\frac{(-1)^{\alpha_1+\cdots+\alpha_\ell}}
{\alpha_1!\cdots\alpha_\ell!}
\prod_{j=1}^\ell(1-s_j)^{\alpha_j}
\E\left[
\prod_{j=1}^\ell(Z_j)_{\alpha_j}
\right]
\nonumber\\
&=
\sum_{\alpha_1,\ldots,\alpha_\ell\geq0}
\prod_{j=1}^\ell
\frac{\left(-\lambda_j(1-s_j)\right)^{\alpha_j}}
{\alpha_j!}
\nonumber\\
&=
\exp\left\{
\sum_{j=1}^\ell\lambda_j(s_j-1)
\right\}.
\label{eq:limit-pgf}
\end{align}
The interchange of expectation and summation is justified by
\eqref{eq:exponential-moment}. The RHS of \eqref{eq:limit-pgf} is the probability
generating function of the product Poisson law
$\nu_{\bm\lambda}$. Hence, $\mu=\nu_{\bm\lambda}$. We have therefore proved $\mu_t\Rightarrow\nu_{\bm\lambda}$. Since $\mathbb N_0^\ell$ is countable and discrete, this implies $\mathrm{TV}(\mu_t,\nu_{\bm\lambda}) \rightarrow0$.
Moreover, by the standard Poisson coupling bound and the tensorization
property of Total Variation distance, $\mathrm{TV}(\nu_{\bm\lambda^{(t)}}, \nu_{\bm\lambda})
\leq \sum_{j=1}^\ell |\lambda_j^{(t)}-\lambda_j| \rightarrow 0$.  Therefore, by the triangle inequality, $\mathrm{TV}(\mu_t,\nu_{\bm\lambda^{(t)}}) \rightarrow0$, which contradicts \eqref{eq:TV-separated}. This proves the lemma.
\end{proof}

Next, we establish a convergence result for a sequence of random finite measures that is used in the proof of Theorem \ref{thm:intersectiondistribution} and Theorem \ref{thm:cubedistribution}.

\begin{lemma}
\label{lem:comPoisson}
Let $\lambda:[0,1]\to[0,\infty)$ be continuous.  Suppose $B_1, B_2, \ldots, B_n$ be i.i.d. $\dBer(\frac{1}{n})$ and $Y_1, Y_2, \ldots, Y_n$ are independent random variables,
which are also independent of $B_1, B_2, \ldots, B_n$, such that 
$Y_{i}\sim \mathrm{Pois}(\lambda (\frac{i}{n} ) )$. Define the random finite measure
\begin{equation*}
\Xi_n:=\sum_{i=1}^n B_{i}Y_{i} \delta_{\frac{i}{n}} . 
\end{equation*}
Then $\Xi_n\Rightarrow\Xi$ as random finite measures on $[0,1]$, where
\begin{equation}
\Xi:=\sum_{x\in\mathcal P} N_x\delta_x,
\label{eq:Xi-limit}
\end{equation}
with $\mathcal P$ a Poisson point process on $[0,1]$ of rate 1, and, conditional on $\mathcal P$, the random variables $\{N_x\}_{x\in\mathcal P}$ are independent with $ N_x\sim\mathrm{Pois}(\lambda(x))$. Equivalently, for every nonnegative continuous function $f$ on $[0,1]$,
\begin{equation}
\mathbb E\exp\left\{-\int f \mathrm{d}\Xi\right\}
=
\exp\left\{
 \int_0^1
 \left(
  \exp\{\lambda(x)(e^{-f(x)}-1)\}-1
 \right)dx
\right\}.
\label{eq:compound-Poisson-Laplace}
\end{equation} 
\end{lemma}

\begin{proof}
Fix a nonnegative $f\in C([0,1])$.  By independence,
\begin{align*}
\mathbb E\exp\left\{-\int f \mathrm{d}\Xi_n\right\} =\prod_{i=1}^n
\mathbb E\exp\left\{ -B_{i}Y_{i}f\left(\frac{i}{n}\right)
\right\} =\prod_{i=1}^n \left[ 1+\frac1n g_f\left(\frac{i}{n}\right)
\right], 
\end{align*}
where $g_f(x) :=\exp\{\lambda(x)(e^{-f(x)}-1)\}-1$. The function $g_f$ is continuous and bounded.  Hence, by a Taylor expansion one has, 
\begin{align*}
\log\mathbb E\exp\left\{-\int f \mathrm{d}\Xi_n\right\} =\sum_{i=1}^n
\log\left[ 1+\frac1n g_f\left(\frac{i}{n}\right)
\right]  &=\frac1n\sum_{i=1}^n g_f\left(\frac{i}{n}\right)
+O\left(\frac1n\right) \nonumber \\
&\rightarrow\int_0^1 g_f(x)  \mathrm{d}x , 
\label{eq:Laplace-Riemann-sum}
\end{align*} 
as $n \rightarrow \infty$. 
Therefore,
\begin{equation}
\mathbb E\exp\left\{-\int f \mathrm{d}\Xi_n\right\}
\rightarrow
\exp\left\{
 \int_0^1
 \left(
  \exp\{\lambda(x)(e^{-f(x)}-1)\}-1
 \right)dx
\right\}.
\label{eq:Laplace-limit}
\end{equation}
This proves \eqref{eq:compound-Poisson-Laplace}. 

We now identify the limit in \eqref{eq:Laplace-limit}. Let $\mathcal P$ and $\{N_x\}_{x\in\mathcal P}$ be as in the statement of the lemma and define $\Xi$ by \eqref{eq:Xi-limit}.  Conditional on $\mathcal P$,
\begin{align*}
\mathbb E\left[
 \exp\left\{-\int f \mathrm{d}\Xi\right\}
 \mathrel{\Big|}\mathcal P
\right]
&=\prod_{x\in\mathcal P}
 \mathbb E e^{- N_xf(x)}\\
&=\prod_{x\in\mathcal P}
 \exp\{\lambda(x)(e^{-f(x)}-1)\}.
\end{align*}
The exponential formula for a Poisson point process now yields exactly
\eqref{eq:compound-Poisson-Laplace}.  Thus, the RHS of
\eqref{eq:Laplace-limit} is the Laplace functional of $\Xi$.
The Laplace-functional convergence criterion for random finite
measures on the compact space $[0,1]$ proves $\Xi_n\Rightarrow\Xi$.  
\end{proof}

\end{document}